\documentclass{amsart}

\usepackage{amssymb}
\usepackage{amsfonts}
\usepackage{amsmath}
\usepackage{amsthm}
\usepackage{graphicx}
\usepackage{mathrsfs}
\usepackage{dsfont}
\usepackage{amscd}
\normalfont
\usepackage[T1]{fontenc}
\usepackage{calligra}
\usepackage{verbatim}
\usepackage[usenames,dvipsnames]{color}
\usepackage[colorlinks=true,linkcolor=Blue,citecolor=Violet]{hyperref}

\makeatletter
\@namedef{subjclassname@2010}{%
  \textup{2010} Mathematics Subject Classification}
\makeatother

\theoremstyle{plain}
\newtheorem{theorem}{Theorem}[section]

\newtheorem{step}{Step}

\newtheorem{corollary}[theorem]{Corollary}

\newtheorem{lemma}[theorem]{Lemma}
\newtheorem{proposition}[theorem]{Proposition}

\newtheorem{theorem-definition}[theorem]{Theorem-Definition}

\theoremstyle{definition}
\newtheorem{definition}[theorem]{Definition}

\newtheorem{notation}[theorem]{Notation}

\newtheorem{convention}[theorem]{Convention}

\theoremstyle{remark}

\newtheorem{remark}[theorem]{Remark}

\numberwithin{equation}{section}

\newcommand{\N}{{\mathds{N}}}
\newcommand{\Z}{{\mathds{Z}}}
\newcommand{\Q}{{\mathds{Q}}}
\newcommand{\R}{{\mathds{R}}}
\newcommand{\C}{{\mathds{C}}}

\newcommand{\D}{{\mathfrak{D}}}
\newcommand{\A}{{\mathfrak{A}}}
\newcommand{\B}{{\mathfrak{B}}}
\newcommand{\M}{{\mathfrak{M}}}

\newcommand{\Lip}{{\mathsf{L}}}

\newcommand{\dist}{{\mathsf{dist}}}

\newcommand{\qpropinquity}[1]{{\mathsf{\Lambda}_{#1}}}

\newcommand{\Kantorovich}[1]{{\mathsf{mk}_{#1}}}

\newcommand{\Haus}[1]{{\mathsf{Haus}_{#1}}}

\newcommand{\StateSpace}{{\mathscr{S}}}

\newcommand{\mongekant}{{Mon\-ge-Kan\-to\-ro\-vich metric}}

\newcommand{\Lqcms}{{\JLL} quantum compact metric space}
\newcommand{\Qqcms}[1]{${#1}$-quasi-Leibniz quantum compact metric space}
\newcommand{\gQqcms}{quasi-Leibniz quantum compact metric space}

\newcommand{\unit}{1}

\newcommand{\sa}[1]{{\mathfrak{sa}\left({#1}\right)}}

\newcommand{\inner}[2]{{\left<{#1},{#2}\right>}}
\newcommand{\JLL}{Lei\-bniz}

\newcommand{\dom}[1]{{\operatorname*{dom}({#1})}}
\newcommand{\diam}[2]{{\mathrm{diam}\left({#1},{#2}\right)}}
\newcommand{\qdiam}[2]{{\mathrm{diam}^\ast\left({#1},{#2}\right)}}
\newcommand{\covn}[3]{{\mathrm{cov}_{\left({#1}\right)}\left({#2}\middle\vert{#3}\right)}}

\newcommand{\bridgereach}[2]{{\varrho\left({#1}\middle|{#2}\right)}}
\newcommand{\bridgeheight}[2]{{\varsigma\left({#1}\middle|{#2}\right)}}

\newcommand{\bridgelength}[2]{{\lambda\left({#1}\middle|{#2}\right)}}

\newcommand{\bridgenorm}[2]{{\mathsf{bn}_{ {#1}  }\left({#2}\right)}}

\newcommand{\alg}[1]{{\mathfrak{#1}}}

\newcommand{\BaireSpace}{{\mathscr{N}}}
\newcommand{\uhf}[1]{{\mathfrak{uhf}\left({#1}\right)}}
\newcommand{\indmor}[2]{{\underrightarrow{#1^{#2}}}}
\newcommand{\af}[1]{{\alg{AF}_{#1}}}

\renewcommand{\geq}{\geqslant}
\renewcommand{\leq}{\leqslant}

\newcommand{\CondExp}[2]{{\mathds{E}\left({#1}\middle\vert{#2}\right)}}

\allowdisplaybreaks[1]

\makeatletter
\newcommand{\vast}{\bBigg@{4}}
\newcommand{\Vast}{\bBigg@{5}}
\makeatother

\begin{document}


\baselineskip=17pt

\title[Quantum Ultrametrics on AF algebras]{Quantum Ultrametrics on AF Algebras and The Gromov-Hausdorff Propinquity}
\author{Konrad Aguilar}
\address{Department of Mathematics \\ University of Denver \\ 2199 S. University Blvd \\ Denver CO 80208}
\email{konrad.aguilar@du.edu}
\urladdr{http://www.math.du.edu/~kaguilar/}
\author{Fr\'{e}d\'{e}ric Latr\'{e}moli\`{e}re}
\address{Department of Mathematics \\ University of Denver \\ 2199 S. University Blvd \\ Denver CO 80208}
\email{frederic@math.du.edu}
\urladdr{http://www.math.du.edu/\symbol{126}frederic/}

\date{\today}

\begin{abstract}
We construct quantum metric structures on unital AF algebras with a faithful tracial state, and prove that for such metrics, AF algebras are limits of their defining inductive sequences of finite dimensional C*-algebras for the quantum propinquity. We then study the geometry, for the quantum propinquity, of three natural classes of AF algebras equipped with our quantum metrics: the UHF algebras, the Effros-Shen AF algebras associated with continued fraction expansions of irrationals, and the Cantor space, on which our construction recovers traditional ultrametrics. We also exhibit several compact classes of AF algebras for the quantum propinquity and show continuity of our family of Lip-norms on a fixed AF algebra. Our work thus brings AF algebras into the realm of noncommutative metric geometry.
\end{abstract}
\maketitle

\setcounter{tocdepth}{1}
\tableofcontents


\section{Introduction}

The Gromov-Hausdorff propinquity \cite{Latremoliere13,Latremoliere13b,Latremoliere14,Latremoliere15,Latremoliere15b}, a family of noncommutative analogues of the Gromov-Hausdorff distance, provides a new framework to study the geometry of classes of C*-algebras, opening new avenues of research in noncommutative geometry. We propose to bring the class of AF algebras into this nascent research project by constructing natural quantum metrics on AF algebras endowed with a faithful tracial state. We prove first that AF algebras endowed with our quantum metrics are indeed limits of some sequence of finite dimensional quantum compact metric spaces for the quantum propinquity. The main application of our AF quantum metrics is the construction of a natural continuous surjection, for the quantum propinquity, from the the space of irrational numbers in $(0,1)$ onto the class of the Effros-Shen AF algebras built in \cite{Effros80b} from continued fraction expansion of irrational numbers --- these AF algebras were of course famously employed by Pimsner-Voiculescu in \cite{PimVoi80a} to complete the classification of the irrational rotation C*-algebras. We also construct another continuous map from the Baire Space onto the class of UHF algebras, and we prove that our construction of quantum metrics, when applied to the Cantor space, recover many standard ultrametrics on that space. Due to this observation, we name our metrics on AF algebras quantum ultrametrics. Moreover, we exploit some of the topological properties of the Baire Space to exhibit many compact sets of AF algebras for the quantum propinquity. 

Various notions of finite dimensional approximations of C*-algebras are found in C*-algebra theory, from nuclearity to quasi-diagonality, passing through exactness, to name a few of the more common notions. They are also a core focus and major source of examples for our research in noncommutative metric geometry. Examples of finite dimensional approximations in the sense of the propinquity include the approximations of quantum tori by fuzzy tori \cite{Latremoliere05,Latremoliere13c} and the full matrix approximations C*-algebras of continuous functions on coadjoint orbits of semisimple Lie groups \cite{Rieffel01,Rieffel10c,Rieffel15}. Moreover, the existence of finite dimensional approximations for quantum compact metric spaces, in the sense of the dual propinquity, were studied in \cite{Latremoliere15}, as part of the discovery by second author of a noncommutative analogue of the Gromov compactness theorem \cite{Gromov81}. 

Among all the types of finite approximations in C*-algebras, Approximately Finite (AF) algebras occupy a special place. Introduced by Bratteli \cite{Bratteli72}, following on the work of Glimm \cite{Glimm60} on UHF algebras, AF algebras are inductive limits, in the category of C*-algebras, of sequences of finite dimensional algebras. Elliott initiated his classification program with AF algebras, and this project brought $K$-theory into the core of C*-algebra theory. Among many problems studied in relation of AF algebras, the fascinating question of when a particular C*-algebra may be embedded into an AF algebra has a long history, with the classification of irrational rotation algebras as a prime example. Thus, the question of making inductive sequences of finite dimensional algebras converge to AF algebras, not only in the sense of inductive limit, but also in terms of the quantum propinquity, is very natural and the seed of this paper.

In order to address this question, we must provide a natural construction of quantum metrics on AF algebras. A quantum metric is provided by a choice of a particular seminorm on a dense subalgebra of a C*-algebra \cite{Rieffel98a, Rieffel99, Latremoliere12b}, called a Lip-norm, which plays an analogue role as the Lipschitz seminorm does in classical metric space theory. The key property that such a seminorm must possess is that its dual must induce a metric on the state space of the underlying C*-algebra which metrizes the weak* topology. This dual metric is a noncommutative analogue of the {\mongekant}, and the idea of this approach to quantum metrics arose in Connes' work \cite{Connes89, Connes} and Rieffel's work. A pair of a unital C*-algebra and a Lip-norm is called a quantum compact metric space, and can be seen as a generalized Lipschitz algebra \cite{weaver99}. However, recent developments in noncommutative metric geometry suggests that some form of relation between the multiplicative structure of C*-algebras and Lip-norms is beneficial \cite{Rieffel09, Rieffel10c, Rieffel11, Rieffel12, Latremoliere13,Latremoliere13b,Latremoliere14,Latremoliere15}. A general form of such a connection is given by the quasi-Leibniz property \cite{Latremoliere15}. As such, we require our quantum metrics on AF algebras to be given by quasi-Leibniz Lip-norms.

Quantum metrics on AF algebras, in turn, allow us to raise further questions, such as the continuity of various important constructions of AF algebras, such that Effros-Shen AF algebras, or Glimm's UHF algebras. These later problems helped guide us to our proposed construction in this paper. We restrict ourselves to the class of unital AF algebras with a faithful tracial state, on which we construct Lip-norms from inductive sequences, the faithful tracial state, and any choice of sequence of positive numbers converging to $0$. The natural sequences to consider are given by the dimension of the C*-algebras constitutive of the inductive sequences. The requirement of a faithful tracial state allows us to construct conditional expectations from which our Lip-norms are built. 

We are then able to prove that, equipped with our metrics, and topologizing the class of quasi-Leibniz quantum compact metric spaces with the quantum propinquity, the class of UHF algebras is the continuous image, in a very natural way, of the Baire space, i.e. the space of sequences of nonzero natural numbers equipped with a standard ultrametric. We then prove that the function which, to any irrational number in $(0,1)$, associates the Effros-Shen AF algebra, becomes continuous as well. This result actually involves the fact that the set of irrational numbers in $(0,1)$ is homeomorphic to the Baire space, and then uses an argument constructed around the continuity of a field of Lip-norms on a well-chosen finite dimensional piece of the Effros-Shen AF algebra. This argument relies, in turn, on computations of certain traces on these finite dimensional algebras, using a $K$-theory argument.

We also prove that our construction for quantum metrics, when applied to the Abelian AF algebra of the $\C$-valued continuous functions on the Cantor space, recover standard ultrametrics on that space. The importance of this observation is that our construction can be seen as a generalization of ultrametrics to the context of AF algebras, which are, informally, zero dimensional quantum compact spaces. We should note that for any two states $\varphi, \psi$ in the state space $\StateSpace(\A)$ of a unital C*-algebra $\A$, the function $t \in [0,1] \mapsto t\varphi + (1-t)\psi$ is a continuous function to $\StateSpace(\A)$ equipped with the weak* topology. Thus, no metric which gives the weak* topology on the state space of a unital C*-algebra can be an ultrametric (as it would imply that the segment between $\varphi$ and $\psi$ would be disconnected, which would be a contradiction). Thus, Lip-norms never induce actual ultrametrics on state spaces, and thus our terminology will not create any confusion, and rather provide interesting candidates of possible quantum ultrametrics.

We also address a question which has proven an interesting challenge in general: the identification of certain compact classes of quasi-Leibniz quantum compact metric spaces for the quantum propinquity. It is unclear that any of the classes of AF algebras which we study in this paper are closed for the quantum propinquity, and moreover the quantum propinquity is not known to be a complete metric, so for a set, being totally bounded and closed together would not be sufficient for the set to be compact in general --- hence the challenge in finding compact classes for the quantum propinquity. The dual propinquity is complete \cite{Latremoliere13b}, which provides a better framework for the study of compactness, but as the dual propinquity is weaker than the quantum propinquity, the question of finding the closure of classes of AF algebras is generally delicate. However, in this paper, we do exhibit natural infinite compact classes of AF algebras for the quantum propinquity, using the topology of the Baire space.

Our construction should be compared with a previous attempt at the construction of natural quantum metrics on AF algebras. In \cite{Antonescu04}, Antonescu and Christensen introduced spectral triples of AF algebras endowed with a faithful state. As their spectral triples are ungraded and their Dirac operators are positive, they only contain metric information. However, the metrics obtained from these spectral triples, when restricted to the Cantor space, are not explicit and do not agree with the usual metrics for that space. Moreover, no convergence result is proven using the metrics associated with these spectral triples when working with noncommutative AF algebras, and it is not clear how one would proceed to prove such results, because the construction of these spectral triples rely on various constants which are not necessarily easy to compute. Our work takes a different perspective: quantum metrics arise naturally from Lip-norms, which need not be defined via spectral triples. Instead, we aim at obtaining natural metrics for which we can actually prove several interesting geometric results, in particular in regards with the quantum propinquity.

Our paper begins with a brief section on the notions of quantum compact metric spaces, quasi-Leibniz Lip-norms, and the quantum Gromov-Hausdorff propinquity. We then construct our Lip-norms for AF algebras. The next two sections establish our main continuity results: first for UHF algebras and second for Effros-Shen AF algebras. We then exhibit some interesting compact classes of AF algebras for the quantum propinquity. In its most general form, our construction of Lip-norms on AF algebras involve a sequence of nonzero natural numbers. We conclude our paper with the proof that our construction is in fact continuous with respect to this parameter.

\section{Quantum Metric Geometry}

We begin our exposition with a brief description of the tools of quantum metric geometry which we will use in this paper. We refer the reader to \cite{Latremoliere15b} for a survey of this area. A by-product of this exposition is also the introduction of notations which we will use throughout our paper.

\begin{notation}
When $E$ is a normed vector space, then its norm will be denoted by $\|\cdot\|_E$ by default.
\end{notation}

\begin{notation}
Let $\A$ be a unital C*-algebra. The unit of $\A$ will be denoted by $\unit_\A$. The state space of $\A$ will be denoted by $\StateSpace(\A)$ while the self-adjoint part of $\A$ will be denoted by $\sa{\A}$. 
\end{notation}

The core objects of noncommutative metric geometry are the quantum compact metric spaces, which are noncommutative generalizations of the algebras of Lipschitz functions over compact metric spaces. The key requirement in the following definition --- that the {\mongekant} metrizes the weak* topology on the state space --- is due to Rieffel. The idea to employ the {\mongekant} as a means to work with noncommutative metrics is due to Connes \cite{Connes89} and was the inspiration for Rieffel's work. The quasi-Leibniz property is the second author's added requirement, itself largely based on Kerr's similar notion of the $F$-Leibniz property \cite{kerr02}, but used for a very different reason --- the quasi-Leibniz property is used to ensure that the Gromov-Hausdorff propinquity has the desired coincidence property, while Kerr used a similar notion to study the completeness of a version of his matricial distance. In \cite{Latremoliere15}, the notion of a quasi-Leibniz Lip-norm is more general than given below, but this will suffice for this paper.

\begin{definition}[\cite{Rieffel98a, Latremoliere13, Latremoliere15}]\label{quasi-Monge-Kantorovich-def}
A {\Qqcms{(C,D)}} $(\A,\Lip)$, for some $C\geq 1$ and $D\geq 0$, is an ordered pair where $\A$ is unital C*-algebra and $\Lip$ is a seminorm defined on some dense Jordan-Lie subalgebra $\dom{\Lip}$ of $\sa{\A}$ such that:
\begin{enumerate}
\item $\{ a \in \sa{\A} : \Lip(a) = 0 \} = \R\unit_\A$,
\item the seminorm $\Lip$ is a \emph{$(C,D)$-quasi-Leibniz Lip-norm}, i.e. for all $a,b \in \dom{\Lip}$:
\begin{equation*}
\max\left\{ \Lip\left(\frac{ab+ba}{2}\right), \Lip\left(\frac{ab-ba}{2i}\right) \right\} \leq C\left(\|a\|_\A \Lip(b) + \|b\|_\A\Lip(a)\right) + D \Lip(a)\Lip(b)\text{,}
\end{equation*}
\item the \emph{\mongekant} defined, for all two states $\varphi, \psi \in \StateSpace(\A)$, by:
\begin{equation*}
\Kantorovich{\Lip} (\varphi, \psi) = \sup\left\{ |\varphi(a) - \psi(a)| : a\in\dom{\Lip}, \Lip(a) \leq 1 \right\}
\end{equation*}
metrizes the weak* topology of $\StateSpace(\A)$,
\item the seminorm $\Lip$ is lower semi-continuous with respect to $\|\cdot\|_\A$.
\end{enumerate}
The seminorm $\Lip$ of a quantum compact metric space $(\A,\Lip)$ is called a \emph{Lip-norm}.
\end{definition}

\begin{convention}
When $\Lip$ is a seminorm defined on some dense subset $F$ of a vector space $E$, we will implicitly extend $\Lip$ to $E$ by setting $\Lip(e) = \infty$ whenever $e \not \in F$.
\end{convention}

Rieffel initiated the systematic study of quantum compact metric space with the following characterization of these spaces, which can be seen as a noncommutative form of the Arz{\'e}la-Ascoli theorem.

\begin{theorem}[\cite{Rieffel98a,Rieffel99, Ozawa05}]\label{Rieffel-thm}
Let $\A$ be a unital C*-algebra and $\Lip$ a seminorm defined on a dense subspace of $\sa{\A}$, such that $\Lip(a) = 0$ if and only if $a\in\R\unit_\A$. The following two assertions are equivalent:
\begin{enumerate}
\item the {\mongekant} $\Kantorovich{\Lip}$ metrizes the weak* topology on $\StateSpace(\A)$,
\item for some state $\mu \in \StateSpace(\A)$, the set:
\begin{equation*}
\left\{ a\in\sa{\A} : \Lip(a)\leq 1, \mu(a) = 0\right\}
\end{equation*}
is totally bounded for $\|\cdot\|_\A$.
\end{enumerate}
\end{theorem}

Our primary interest in developing a theory of quantum metric spaces is the introduction of various hypertopologies on classes of such spaces, thus allowing us to study the geometry of classes of C*-algebras and perform analysis on these classes. A classical model for our hypertopologies is given by the Gromov-Hausdorff distance. While several noncommutative analogues of the Gromov-Hausdorff distance have been proposed --- most importantly Rieffel's original construction of the quantum Gromov-Hausdorff distance \cite{Rieffel00} --- we shall work with a particular metric introduced by the second author. This metric, known as the quantum propinquity, is designed to be best suited to {\gQqcms s}, and in particular, is zero between two such spaces if and only if they are isometrically isomorphic (unlike Rieffel's distance). We now propose a summary of the tools needed to compute upper bounds on this metric.

\begin{definition}
The \emph{$1$-level set} $\StateSpace_1(\D|\omega)$ of an element $\omega$ of a unital C*-algebra $\D$ is:
\begin{equation*}
\left\{ \varphi \in \StateSpace(\D) : \varphi((1-\omega^\ast\omega))=\varphi((1-\omega \omega^\ast)) = 0 \right\}\text{.}
\end{equation*}
\end{definition}

\begin{definition}\label{bridge-def}
A \emph{bridge} from $\A$ to $\B$, where $\A$ and $\B$ are unital C*-algebras, is a quadruple $(\D,\pi_\A,\pi_\B,\omega)$ where:
\begin{enumerate}
\item $\D$ is a unital C*-algebra,
\item the element $\omega$, called the \emph{pivot} of the bridge, satisfies $\omega\in\D$ and $\StateSpace_1(\D|\omega) \not=\emptyset$,
\item $\pi_\A : \A\hookrightarrow \D$ and $\pi_\B : \B\hookrightarrow\D$ are unital *-monomorphisms.
\end{enumerate}
\end{definition}

There always exists a bridge between any two arbitrary {\gQqcms s} \cite{Latremoliere13,Latremoliere15}. A bridge allows us to define a numerical quantity which estimates, for this given bridge, how far our {\gQqcms s} are. This quantity, called the length of the bridge, is constructed using two other quantities we now define. 

In the next few definitions, we denote by $\Haus{\mathrm{d}}$ the \emph{Hausdorff (pseudo)distance} induced by a (pseudo)distance $\mathrm{d}$ on the compact subsets of a (pseudo)metric space $(X,\mathrm{d})$ \cite{Hausdorff}.

The height of a bridge assesses the error we make by replacing the state spaces of the {\Lqcms s} with the image of the $1$-level set of the pivot of the bridge, using the ambient {\mongekant}. 

\begin{definition}
Let $(\A,\Lip_\A)$ and $(\B,\Lip_\B)$ be two {\gQqcms s}. The \emph{height} $\bridgeheight{\gamma}{\Lip_\A,\Lip_\B}$ of a bridge $\gamma = (\D,\pi_\A,\pi_\B,\omega)$ from $\A$ to $\B$, and with respect to $\Lip_\A$ and $\Lip_\B$, is given by:
\begin{equation*}
\max\left\{ \Haus{\Kantorovich{\Lip_\A}}(\StateSpace(\A), \pi_\A^\ast(\StateSpace_1(\D|\omega))), \Haus{\Kantorovich{\Lip_\B}}(\StateSpace(\B), \pi_\B^\ast(\StateSpace_1(\D|\omega))) \right\}\text{,}
\end{equation*}
where $\pi_\A^{\ast}$ and $\pi_\B^\ast$ are the dual maps of $\pi_\A$ and $\pi_\B$, respectively.
\end{definition}

The second quantity measures how far apart the images of the balls for the Lip-norms are in $\A\oplus\B$; to do so, they use a seminorm on $\A\oplus\B$ built using the bridge:
\begin{definition}
Let $(\A,\Lip_\A)$ and $(\B,\Lip_\B)$ be two unital C*-algebras. The \emph{bridge seminorm} $\bridgenorm{\gamma}{\cdot}$ of a bridge $\gamma = (\D,\pi_\A,\pi_\B,\omega)$ from $\A$ to $\B$ is the seminorm defined on $\A\oplus\B$ by:
\begin{equation*}
\bridgenorm{\gamma}{a,b} = \|\pi_\A(a)\omega - \omega\pi_\B(b)\|_\D
\end{equation*}
for all $(a,b) \in \A\oplus\B$.
\end{definition}

We implicitly identify $\A$ with $\A\oplus\{0\}$ and $\B$ with $\{0\}\oplus\B$ in $\A\oplus\B$ in the next definition, for any two spaces $\A$ and $\B$.

\begin{definition}
Let $(\A,\Lip_\A)$ and $(\B,\Lip_\B)$ be two {\gQqcms s}. The \emph{reach} $\bridgereach{\gamma}{\Lip_\A,\Lip_\B}$ of a bridge $\gamma = (\D,\pi_\A,\pi_\B,\omega)$ from $\A$ to $\B$, and with respect to $\Lip_\A$ and $\Lip_\B$, is given by:
\begin{equation*}
\Haus{\bridgenorm{\gamma}{\cdot}}\left( \left\{a\in\sa{\A} : \Lip_\A(a)\leq 1\right\} , \left\{ b\in\sa{\B} : \Lip_\B(b) \leq 1 \right\}  \right) \text{.}
\end{equation*}
\end{definition}

We thus choose a natural quantity to synthesize the information given by the height and the reach of a bridge:

\begin{definition}
Let $(\A,\Lip_\A)$ and $(\B,\Lip_\B)$ be two {\gQqcms s}. The \emph{length} $\bridgelength{\gamma}{\Lip_\A,\Lip_\B}$ of a bridge $\gamma = (\D,\pi_\A,\pi_\B,\omega)$ from $\A$ to $\B$, and with respect to $\Lip_\A$ and $\Lip_\B$, is given by:
\begin{equation*}
\max\left\{\bridgeheight{\gamma}{\Lip_\A,\Lip_\B}, \bridgereach{\gamma}{\Lip_\A,\Lip_\B}\right\}\text{.}
\end{equation*}
\end{definition}

While a natural approach, defining the quantum propinquity as the infimum of the length of all possible bridges between two given {\Qqcms{(C,D)}s}, for some fixed $C\geq 1$ and $D\geq 0$, does not lead to a distance, as the triangle inequality may not be satisfied. Instead, a more subtle road must be taken, as exposed in details in \cite{Latremoliere13}. The following theorem hides these complications and provide a summary of the conclusions of \cite{Latremoliere13} relevant for our work:

\begin{theorem-definition}[\cite{Latremoliere13, Latremoliere15}]\label{def-thm}
Fix $C\geq 1$ and $D \geq 0$. Let $\mathcal{QQCMS}_{C,D}$ be the class of all {\Qqcms{(C,D)}s}. There exists a class function $\qpropinquity{C,D}$ from $\mathcal{QQCMS_{C,D}}\times\mathcal{QQCMS}_{C,D}$ to $[0,\infty) \subseteq \R$ such that:
\begin{enumerate}
\item for any $(\A,\Lip_\A), (\B,\Lip_\B) \in \mathcal{QQCMS}_{C,D}$ we have:
\begin{equation*}
 \qpropinquity{C,D}((\A,\Lip_\A),(\B,\Lip_\B)) \leq \max\left\{\diam{\StateSpace(\A)}{\Kantorovich{\Lip_\A}}, \diam{\StateSpace(\B)}{\Kantorovich{\Lip_\B}}\right\}\text{,}
\end{equation*}
\item for any $(\A,\Lip_\A), (\B,\Lip_\B) \in \mathcal{QQCMS}_{C,D}$ we have:
\begin{equation*}
0\leq \qpropinquity{C,D}((\A,\Lip_\A),(\B,\Lip_\B)) = \qpropinquity{C,D}((\B,\Lip_\B),(\A,\Lip_\A))
\end{equation*}
\item for any $(\A,\Lip_\A), (\B,\Lip_\B), (\alg{C},\Lip_{\alg{C}}) \in \mathcal{QQCMS}_{C,D}$ we have:
\begin{equation*}
\qpropinquity{C,D}((\A,\Lip_\A),(\alg{C},\Lip_{\alg{C}})) \leq \qpropinquity{C,D}((\A,\Lip_\A),(\B,\Lip_\B)) + \qpropinquity{C,D}((\B,\Lip_\B),(\alg{C},\Lip_{\alg{C}}))\text{,}
\end{equation*}
\item for all  for any $(\A,\Lip_\A), (\B,\Lip_\B) \in \mathcal{QQCMS}_{C,D}$ and for any bridge $\gamma$ from $\A$ to $\B$, we have:
\begin{equation*}
\qpropinquity{C,D}((\A,\Lip_\A), (\B,\Lip_\B)) \leq \bridgelength{\gamma}{\Lip_\A,\Lip_\B}\text{,}
\end{equation*}
\item for any $(\A,\Lip_\A), (\B,\Lip_\B) \in \mathcal{QQCMS}_{C,D}$, we have:
\begin{equation*}
\qpropinquity{C,D}((\A,\Lip_\A),(\B,\Lip_\B)) = 0
\end{equation*}
if and only if $(\A,\Lip_\A)$ and $(\B,\Lip_\B)$ are isometrically isomorphic, i.e. if and only if there exists a *-isomorphism $\pi : \A \rightarrow\B$ with $\Lip_\B\circ\pi = \Lip_\A$, or equivalently there exists a *-isomorphism $\pi : \A \rightarrow\B$ whose dual map $\pi^\ast$ is an isometry from $(\StateSpace(\B),\Kantorovich{\Lip_\B})$ into $(\StateSpace(\A),\Kantorovich{\Lip_\A})$,
\item if $\Xi$ is a class function from $\mathcal{QQCMS}_{C,D}\times \mathcal{QQCMS}_{C,D}$ to $[0,\infty)$ which satisfies Properties (2), (3) and (4) above, then:
\begin{equation*}
\Xi((\A,\Lip_\A), (\B,\Lip_\B)) \leq \qpropinquity{C,D}((\A,\Lip_\A),(\B,\Lip_\B))
\end{equation*}
 for all $(\A,\Lip_\A)$ and $(\B,\Lip_\B)$ in $\mathcal{QQCMS}_{C,D}$
\end{enumerate}
\end{theorem-definition}

Thus, for a fixed choice of $C\geq 1$ and $D\geq 0$, the quantum propinquity is the largest pseudo-distance on the class of {\Qqcms{(C,D)}s} which is bounded above by the length of any bridge between its arguments; the remarkable conclusion of \cite{Latremoliere13} is that this pseudo-metric is in fact a metric up to isometric isomorphism. The quantum propinquity was originally devised in the framework on Leibniz quantum compact metric spaces (i.e. for the case $C=1$ and $D=0$), and as seen in \cite{Latremoliere15}, can be extended to many different classes of quasi-Leibniz compact quantum metric spaces.

Moreover, we showed in \cite{Latremoliere13} that we can compare the quantum propinquity to natural metrics.

\begin{theorem}[\cite{Latremoliere13}]
If $\dist_q$ is Rieffel's quantum Gromov-Hausdorff distance \cite{Rieffel00}, then for any pair $(\A,\Lip_\A)$ and $(\B,\Lip_\B)$ of {\gQqcms s}, we have:
\begin{equation*}
\dist_q((\A,\Lip_\A),(\B,\Lip_\B)) \leq \qpropinquity{}((\A,\Lip_\A),(\B,\Lip_\B)) \text{.}
\end{equation*}
Moreover, for any compact metric space $(X,\mathrm{d}_X)$, let $\Lip_{\mathrm{d}_X}$ be the Lipschitz seminorm induced on the C*-algebra $C(X)$ of $\C$-valued continuous functions on $X$ by $\mathrm{d}_X$. Note that $(C(X),\Lip_{\mathrm{d}_X})$ is a {\Lqcms}. Let $\mathfrak{C}$ be the class of all compact metric spaces. For any $(X,\mathrm{d}_x), (Y,\mathrm{d_Y}) \in \mathfrak{C}$, we have:
\begin{equation*}
\qpropinquity{}\left(\left(C(X),\Lip_{\mathrm{d}_X}\right), \left(C(Y),\Lip_{\mathrm{d}_Y}\right)\right) \leq \mathrm{GH}((X,\mathrm{d}_X),(Y,\mathrm{d}_Y))
\end{equation*}
where $\mathrm{GH}$ is the Gromov-Hausdorff distance \cite{Gromov81, Gromov}. 

Furthermore, the class function $\Upsilon : (X,\mathrm{d}_X) \in \mathfrak{C} \mapsto \left(C(X),\Lip_{\mathrm{d}_X}\right)$ is a homeomorphism, where the topology on $\mathfrak{C}$  is given by the Gromov-Hausdorff distance $\mathrm{GH}$, and the topology on the image of $\Upsilon$ (as a subclass of the class of all {\Lqcms s}) is given by the quantum propinquity $\qpropinquity{}$.
\end{theorem}

As we noted, the construction and many more information on the quantum Gromov-Hausdorff propinquity can be found in our original paper \cite{Latremoliere13} on this topic, as well as in our survey \cite{Latremoliere15b}. The extension of our original work to the quasi-Leibniz setting can be found in \cite{Latremoliere15}. Two very important examples of nontrivial convergences for the quantum propinquity are given by quantum tori and their finite dimensional approximations, as well as certain metric perturbations \cite{Latremoliere05,Latremoliere13c,Latremoliere15c} and by matrix approximations of the C*-algebras of coadjoint orbits for semisimple Lie groups \cite{Rieffel10c, Rieffel11,Rieffel15}. Moreover, the quantum propinquity is, in fact, a special form of the dual Gromov-Hausdorff propinquity \cite{Latremoliere13b, Latremoliere14, Latremoliere15}, which is a complete metric, up to isometric isomorphism, on the class of {\Lqcms s}, and which extends the topology of the Gromov-Hausdorff distance as well. Thus, as the dual propinquity is dominated by the quantum propinquity \cite{Latremoliere13b}, we conclude that \emph{all the convergence results in this paper are valid for the dual Gromov-Hausdorff propinquity as well.}

The present paper establishes new examples of convergence for the quantum propinquity by constructing quantum metrics on certain AF algebras. All our quantum metrics will be {\Qqcms{(2,0)}s}. Thus, we will simplify our notation as follows:
\begin{convention}
In this paper, $\qpropinquity{}$ will be meant for $\qpropinquity{2,0}$.
\end{convention}

\section{AF algebras as Quasi-Leibniz Quantum Compact Metric Spaces}

We begin by observing that conditional expectations allow us to define $(2,0)$-quasi-Leibniz seminorms on C*-algebras.

\begin{definition}
A \emph{conditional expectation} $\CondExp{\cdot}{\B}:\A\rightarrow\B$ onto $\B$, where $\A$ is a C*-algebra and $\B$ is a C*-subalgebra of $\A$, is a linear positive map of norm $1$ such that for all $b, c \in \B$ and $a\in \A$ we have:
\begin{equation*}
\CondExp{b a c}{\B} = b \CondExp{a}{\B} c\text{.}
\end{equation*}
\end{definition}

\begin{lemma}\label{expt-seminorm-lemma}
Let $\A$ be a C*-algebra and $\B \subseteq\A$ be a C*-subalgebra of $\A$. If $\CondExp{\cdot}{\B} : \A\mapsto \B$ is a conditional expectation onto $\B$, then the seminorm:
\begin{equation*}
\mathrm{S} : a\in \A \mapsto \|a-\CondExp{a}{\B}\|_\A
\end{equation*}
is a $(2,0)$-quasi-Leibniz seminorm. 
\end{lemma}

\begin{proof}
Let $a,b \in \A$. We have:
\begin{equation*}
\begin{split}
\mathrm{S}(ab) &= \|ab-\CondExp{ab}{\B}\|_\A \\
&\leq \|ab - a\CondExp{b}{\B}\|_\A + \|a\CondExp{b}{\B} - \CondExp{ab}{\B}\|_\A\\
&\leq \|a\|_\A\|b-\CondExp{b}{\B}\|_\A \\
&\quad + \|a\CondExp{b}{\B} - \CondExp{a\CondExp{b}{\B}}{\B} + \CondExp{a(\CondExp{b}{\B}-b)}{\B} \|_\A\\
&\leq \|a\|_\A\|b-\CondExp{b}{\B}\|_\A + \|a-\CondExp{a}{\B})\|_\A\|\CondExp{b}{\B}\|_\A \\
&\quad + \|\CondExp{a(b-\CondExp{b}{\B})}{\B}\|_\A\\
&\leq \|a\|_\A\|b-\CondExp{b}{\B}\|_\A + \|a-\CondExp{a}{\B}\|_\A\|\CondExp{b}{\B}\|_\A \\
&\quad + \|a\|_\A\|b-\CondExp{b}{\B}\|_\A\\
&\leq 2\|a\|_\A\|b-\CondExp{b}{\B}\|_\A + \|a-\CondExp{a}{\B})\|_\A\|b\|_\A \\
&\leq 2\left(\|a\|_\A \mathrm{S}(b) + \|b\|_\A\mathrm{S}(a)\right) \text{.}
\end{split}
\end{equation*}
This proves our lemma.
\end{proof}

Note that the seminorms defined by Lemma (\ref{expt-seminorm-lemma}) are zero exactly on the range of the conditional expectation. Now, our purpose is to define quasi-Leibniz Lip-norms on AF C*-algebras using Lemma (\ref{expt-seminorm-lemma}) and a construction familiar in Von Neumann theory, which we recall here for our purpose.

We shall work with unital AF algebras \cite{Bratteli79} endowed with a faithful tracial state. Any unital AF algebra admits at least one tracial state \cite[Proposition 3.4.11]{Lin2001}, and thus simple AF algebras admit at least one faithful tracial state. In fact, the space of tracial states of unital simple AF algebras can be any Choquet simplex \cite{Goodearl77, Blackadar80}. On the other hand, a unital AF algebra has a faithful trace if, and only if it is a C*-subalgebra of a unital simple AF algebra \cite[Corollary 4.3]{Lin08}. Examples of unital AF algebras without a faithful trace can be obtained as essential extensions of the algebra of compact operators of a separable Hilbert space by some full matrix  algebra. Thus, our context could be stated as the study of certain Lip-norms on unital AF algebras which can be embedded in unital simple AF algebras.

Our main construction of Lip-norms on unital AF algebras with a faithful tracial state is summarized in the following theorem.

\begin{notation}
Let $\mathcal{I} = (\A_n,\alpha_n)_{n\in\N}$ be an inductive sequence with limit $\A=\varinjlim \mathcal{I}$. We denote the canonical *-morphisms $\A_n \rightarrow\A$ by $\indmor{\alpha}{n}$ for all $n\in\N$.
\end{notation}

\begin{convention}
We assume that for all the inductive sequences $(\A_n,\alpha_n)_{n\in\N}$ in this paper, the C*-algebras $\A_n$ are unital and the *-morphisms $\alpha_n$ are unital and injective for all $n\in\N$. 
\end{convention}

\begin{theorem}\label{AF-lip-norms-thm}
Let $\A$ be an AF algebra endowed with a faithful tracial state $\mu$. Let $\mathcal{I} = (\A_n,\alpha_n)_{n\in\N}$ be an inductive sequence of finite dimensional C*-algebras with C*-inductive limit $\A$, with $\A_0 = \C$ and where $\alpha_n$ is unital and injective for all $n\in\N$.

Let $\pi$ be the GNS representation of $\A$ constructed from $\mu$ on the space $L^2(\A,\mu)$.

For all $n\in\N$, let:
\begin{equation*}
\CondExp{\cdot}{\indmor{\alpha}{n}(\A_n)} : \A\rightarrow\A
\end{equation*}
be the unique conditional expectation of $\A$ onto the canonical image $\indmor{\alpha}{n}\left(\A_n\right)$ of $\A_n$ in $\A$, and such that $\mu\circ\CondExp{\cdot}{\indmor{\alpha}{n}(\A_n)} = \mu$.

Let $\beta: \N\rightarrow (0,\infty)$ have limit $0$ at infinity. If, for all $a\in\sa{\A}$, we set:
\begin{equation*}
\Lip_{\mathcal{I},\mu}^\beta(a) = \sup\left\{\frac{\left\|a - \CondExp{a}{\indmor{\alpha}{n}(\A_n)}\right\|_\A}{\beta(n)} : n \in \N \right\}
\end{equation*}
then $\left(\A,\Lip_{\mathcal{I},\mu}^\beta\right)$ is a {\Qqcms{2}}. Moreover for all $n\in\N$:
\begin{equation*}
\qpropinquity{}\left(\left(\A_n,\Lip_{\mathcal{I},\mu}^\beta\circ\indmor{\alpha}{n} \right), \left(\A,\Lip_{\mathcal{I},\mu}^\beta \right)\right) \leq \beta(n)
\end{equation*}
and thus:
\begin{equation*}
\lim_{n\rightarrow\infty} \qpropinquity{}\left(\left(\A_n,\Lip_{\mathcal{I},\mu}^\beta\circ\indmor{\alpha}{n}\right), \left(\A,\Lip_{\mathcal{I},\mu}^\beta\right)\right) = 0\text{.}
\end{equation*}
\end{theorem}

\begin{proof}
To begin with, we note that, from the standard GNS construction, we have the following:
\begin{enumerate}
\item since $\mu$ is faithful, the map $\xi : a \in \A \mapsto a\in L^2(\A,\mu)$ is injective,
\item since $\Vert\xi(a)\Vert_{L^2(\A,\mu)} =\sqrt{\mu(a^\ast a)} \leq \|a\|_\A$ for all $a\in\A$, the map $\xi$ is a continuous (weak) contraction,
\item by construction, $\xi(ab) = \pi(a)\xi(b)$ for all $a,b \in \A$,
\item if $\omega$ is $\xi(\unit_\A)$, then $\omega$ is cyclic and $\xi(a) = \pi(a)\omega$.
\end{enumerate}

Let $n\in\N$. We denote the canonical unital *-monomorphism from $\A_n$ into $\A$ by $\indmor{\alpha}{n}$. Thus $\xi\circ\indmor{\alpha}{n} : \A_n \rightarrow L^2(\A,\mu)$ is a linear, weakly contractive injection. Since $\A_n$ is finite dimension, $\xi\circ\indmor{\alpha}{n}(\A_n)$ is a closed subspace of $L^2(\A,\mu)$.  Let $P_n$ be the orthogonal projection from $L^2(\A,\mu)$ onto $\xi\circ\indmor{\alpha}{n}(\A_n)$.

We thus note that for all $a\in\A$, we have $P_n (\xi(a)) \in \xi\circ\indmor{\alpha}{n}(\A_n)$, thus, since $\xi$ is injective, there exists a unique $\mathds{E}_n(a) \in \indmor{\alpha}{n}(\A_n)$ with $\xi(\mathds{E}_n(a)) = P_n(\xi(a))$. 

\begin{step}
We begin by checking that the map $\mathds{E}_n : \A\rightarrow \indmor{\alpha}{n}(\A_n)$ is the conditional expectation $\CondExp{\cdot}{\indmor{\alpha}{n}(\A_n)}$ of $\A$ onto $\indmor{\alpha}{n}(\A_n)$ which preserves the state $\mu$.
\end{step}

To begin with, if $a\in\A_n$ then $P_n\xi(\indmor{\alpha}{n}(a)) = \xi(\indmor{\alpha}{n}(a))$ so $\mathds{E}_n(a) = \indmor{\alpha}{n}(a)$. Thus $\mathds{E}_n$ is onto $\indmor{\alpha}{n}(\A_n)$, and restricts to the identity on $\indmor{\alpha}{n}(\A_n)$.

We now prove that $P_n$ commutes with $\pi(a)$ for all $a\in\indmor{\alpha}{n}(\A_n)$. Let $a\in\indmor{\alpha}{n}(\A_n)$. We note that if $b \in \indmor{\alpha}{n}(\A_n)$ then $\pi(a)\xi(b) = \xi(ab) \in \xi(\indmor{\alpha}{n}(\A_n))$ since $\indmor{\alpha}{n}(\A_n)$ is a subalgebra of $\A$. Thus $\pi(a)\left(\xi(\indmor{\alpha}{n}(\A_n))\right) \subseteq \xi(\indmor{\alpha}{n}(\A_n))$. Since $\indmor{\alpha}{n}(\A_n)$ is closed under the adjoint operation, and $\pi$ is a *-representation, we have $\pi(a^\ast)\xi(\indmor{\alpha}{n}(\A_n))\subseteq \xi(\indmor{\alpha}{n}(\A_n))$. Thus, if we let $x \in \xi(\indmor{\alpha}{n}(\A_n))^\perp$ and $y \in \xi(\indmor{\alpha}{n}(\A_n))$, we then have:
\begin{equation*}
\inner{\pi(a)x}{y} = \inner{x}{\pi(a^\ast)y} = 0\text{,}
\end{equation*}
i.e. $\pi(a)(\xi(\indmor{\alpha}{n}(\A_n))^\perp) \subseteq \xi(\indmor{\alpha}{n}(\A_n))^\perp$. Consequently, if $x\in L^2(\A,\mu)$, writing $x = P_nx + P^\perp_n x$, we have:
\begin{equation*}
P_n\pi(a) x = P_n\pi(a)P_nx + P_n\pi(a)P_n^\perp x = \pi(a)P_n x\text{.}
\end{equation*}
In other words, $P_n$ commutes with $\pi(a)$ for all $a\in\indmor{\alpha}{n}(\A_n)$.

As a consequence, for all $a\in\indmor{\alpha}{n}(\A_n)$ and $b\in\A$:
\begin{equation*}
\xi(\mathds{E}_n(ab)) = P_n\pi(a)\xi(b) = \pi(a)P_n\xi(b) = \pi(a)\xi(\mathds{E}_n(b)) = \xi(a\mathds{E}_n(b))\text{.}
\end{equation*}
Thus $\mathds{E}_n(ab) = a\mathds{E}_n(b)$ for all $a\in\indmor{\alpha}{n}(\A_n)$ and $b\in\A$. 

We now wish to prove that $\mathds{E}_n$ is a *-linear map. Let $J : \xi(x) \mapsto \xi\left(x^\ast\right)$. The key observation is that, since $\mu$ is a trace:
\begin{equation*}
\inner{J\xi(x)}{J\xi(y)} = \mu(y x^\ast) = \mu(x^\ast y) =  \inner{x}{y} 
\end{equation*}
hence $J$ is an conjugate-linear isometry and can be extended to $L^2(\A,\mu)$. It is easy to check that $J$ is surjective, as it has a dense range and is isometric, in fact $J = J^\ast = J^{-1}$. This is the only point where we use that $\mu$ is a trace.

We now check that $P_n$ and $J$ commute. To begin with, we note that:
\begin{equation*}
\begin{split}
( JP_n J)( JP_n J) = JP_nJ
\end{split}
\end{equation*}
and thus the self-adjoint operator $JP_n J$ is a projection. Let $a\in\A$. Then:
\begin{equation*}
JP_nJ\xi(a) = JP_n\xi(a^\ast) = J\xi(\mathds{E}_n(a^\ast)) = \xi(\mathds{E}_n(a^\ast)^\ast) \in \xi(\indmor{\alpha}{n}(\A_n))\text{.}
\end{equation*}
Thus $JP_nJ = P_n$, so $P_n$ and $J$ commute since $J^2 = \unit_{\B(L^2(\A,\mu))}$. 

Consequently for all $a\in\A$:
\begin{equation*}
\xi(\mathds{E}_n(a^\ast)) = P_n  \xi(a^\ast) = P_n J \xi(a) = J P_n \xi(a) = J \xi(\mathds{E}_n(a)) = \xi(\mathds{E}_n(a)^\ast) \text{,}
\end{equation*}
so $\mathds{E}_n(a^\ast) = \mathds{E}_n(a)^\ast$.

In particular, we note that for all $a\in\A$ and $b, c\in \indmor{\alpha}{n}(\A_n)$ we have:
\begin{equation*}
\mathds{E}_n(b a c) = b \mathds{E}_n(a c) = b \mathds{E}_n(c^\ast a^\ast)^\ast = b(c^\ast \mathds{E}_n(a)^\ast)^\ast) = b\mathds{E}_n(a)c \text{.}
\end{equation*}

To prove that $\mathds{E}_n$ is a positive map, we begin by checking that it preserves the state $\mu$. First note that $\unit_\A \in \indmor{\alpha}{n}(\A_n)$ so $\omega\in\xi(\indmor{\alpha}{n}(\A_n))$, and thus $P_n\omega=\omega$. Thus for all $a\in\A$:

\begin{equation*}
\begin{split}
\mu(\mathds{E}_n(a)) &= \inner{\pi(\mathds{E}_n(a))\omega}{\omega} \\
&= \inner{\xi(\mathds{E}_n(a))}{\omega} = \inner{P_n\xi(a)}{\omega}\\
&= \inner{\xi(a)}{P_n \omega} = \inner{\pi(a)\omega}{P_n\omega}\\
&= \inner{\pi(a)\omega}{\omega} = \mu(a) \text{.}
\end{split}
\end{equation*}
Thus $\mathds{E}_n$ preserves the state $\mu$. More generally, using the conditional expectation property, for all $b,c \in \indmor{\alpha}{n}(\A_n)$ and $a\in\A$:
\begin{equation*}
\mu(b \mathds{E}_n(a) c) = \mu(b a c)\text{.}
\end{equation*}

We now prove that $\mathds{E}_n$ is positive. First, $\mu$ restricts to a faithful state of $\indmor{\alpha}{n}(\A_n)$ and $L^2(\indmor{\alpha}{n}(\A_n),\mu)$ is given canonically by $\xi(\indmor{\alpha}{n}(\A_n))$. Let now $a\in\sa{\A}$ with $a\geq 0$. We now have for all $b \in \indmor{\alpha}{n}(\A_n)$ that:
\begin{equation*}
\inner{\mathds{E}_n(a)\xi(b)}{\xi(b)} = \mu(b^\ast \mathds{E}_n(a) b) = \mu(b^\ast a b) \geq 0\text{.}
\end{equation*}
Thus the operator $\mathds{E}_n(a)$ is positive in $\indmor{\alpha}{n}(\A_n)$. Thus $\mathds{E}_n$ is positive.

Since $\mathds{E}_n$ restricts to the identity on $\indmor{\alpha}{n}(\A_n)$, this map is of norm at least one. Now, let $a\in \sa{\A}$ and $\varphi \in \StateSpace(\A)$. Then $\varphi \circ\mathds{E}_n$ is a state of $\A$ since $\mathds{E}_n$ is positive and unital. Thus $\left|\varphi\circ\mathds{E}_n(a)\right| \leq \|a\|_\A$. As $\mathds{E}_n(\sa{\A}) \subseteq \sa{\A}$, we have:
\begin{equation}\label{cp-sa-eq1}
\forall a\in \sa{\A}\quad \left\|\mathds{E}_n(a)\right\|_\A = \sup\left\{ |\varphi\circ\mathds{E}_n(a)| : \varphi \in \StateSpace(\A) \right\} \leq \|a\|_\A\text{.}
\end{equation}
Thus $\mathds{E}_n$ restricted to $\sa{\A}$ is a linear map of norm $1$.

On the other hand, for all $a\in\A$, we have:
\begin{equation*}
\begin{split}
0 &\leq \mathds{E}_n\left ( \left(a-\mathds{E}_n(a)\right)^\ast\left(a-\mathds{E}_n(a)\right)\right ) \\
&= \mathds{E}_n\left(a^\ast a\right) - \mathds{E}_n\left(\mathds{E}_n(a)^\ast a\right) - \mathds{E}_n\left(a^\ast \mathds{E}_n(a)\right) + \mathds{E}_n\left(\mathds{E}_n(a)^\ast \mathds{E}_n(a)\right)\\
&= \mathds{E}_n\left(a^\ast a\right) - \mathds{E}_n(a)^\ast \mathds{E}_n(a) \text{.}
\end{split}
\end{equation*}
Thus for all $a\in\A$ we have:
\begin{equation*}
\begin{split}
\left\| \mathds{E}_n(a) \right\|_\A^2 &= \left\| \mathds{E}_n(a)^\ast \mathds{E}_n(a) \right\|_\A \\
&\leq \left\| \mathds{E}_n(a^\ast a) \right\|_\A \\
&\leq \|a^\ast a\|_\A = \|a\|_\A^2 \text{ by Inequality (\ref{cp-sa-eq1}).}
\end{split}
\end{equation*}
Thus $\mathds{E}_n$ has norm $1$. We conclude that $\mathds{E}_n$ is a conditional expectation onto $\indmor{\alpha}{n}(\A_n)$ which preserves $\mu$.

Now, assume $T : \A\rightarrow \indmor{\alpha}{n}(\A_n)$ is a unital conditional expectation such that $\mu\circ T = \mu$. As before, we have:
\begin{equation*}
\mu(bT(a)c) = \mu (b a c)
\end{equation*}
for all $a\in\A$ and $b,c \in \indmor{\alpha}{n}(\A_n)$. Thus, for all $x,y \in L^2(\indmor{\alpha}{n}(\A_n),\mu)$ and for all $a\in\A$, we compute:
\begin{equation*}
\inner{T(a)x}{y} = \mu(y^\ast T(a) x) = \mu(y^\ast a x) = \mu(y^\ast \mathds{E}_n(a) x) = \inner{\mathds{E}_n(a)x}{y}
\end{equation*}
and thus $\mathds{E}_n(a) = T(a)$ for all $a\in\A$. So $\mathds{E}_n$ is the unique conditional expectation from $\A$ onto $\indmor{\alpha}{n}(\A_n)$ which preserves $\mu$.

\begin{step}
The seminorm $\Lip_{\mathcal{I},\mu}^\beta$ is a $(2,0)$-quasi-Leibniz Lip-norm on $\A$, and $\mathds{E}_n$ is weakly contractive for $\Lip_{\mathcal{I},\mu}^\beta$ and for all $n\in\N$.
\end{step}

We conclude from Lemma (\ref{expt-seminorm-lemma}) and from Step 1 that $\Lip_{\mathcal{I},\mu}^\beta$ is a $(2,0)$-quasi-Leibniz seminorm.

If $a\in\sa{\A}$ and $\Lip_{\mathcal{I},\mu}^\beta(a) = 0$ then $\|a-\mathds{E}_0(a)\|_\A = 0$ and thus $a\in\sa{\indmor{\alpha}{0}(\C)} = \R\unit_\A$. 

We also note that if $a\in\sa{\A}$ with $\Lip_{\mathcal{I},\mu}^\beta(a)\leq 1$ then $\|a-\mathds{E}_0(a)\|_\A \leq \beta(0)$. Note that $\mathds{E}_0(a) = \mu(a)\unit_\A$ as $\mathds{E}_0$ preserves $\mu$.

For all $n, p \in \N$ we have $\mathds{E}_p\circ\mathds{E}_n = \mathds{E}_{\min\{n,p\}}$ by construction (since $P_n P_p = P_{\min\{n,p\}}$). Thus, if $n\leq p$ and $a\in\sa{\A}$ then:
\begin{equation}\label{AF-lip-norms-thm-eq1}
\|\mathds{E}_n(a)-\mathds{E}_p(\mathds{E}_n(a)))\|_\A = 0\text{.}
\end{equation}
In particular, we conclude that the dense Jordan-Lie subalgebra $\sa{\bigcup_{n\in\N}\indmor{\alpha}{n}(\A_n)}$ of $\sa{\A}$ is included in the domain $\dom{\Lip_{\mathcal{I},\mu}^\beta}$ of $\Lip_{\mathcal{I},\mu}^\beta$ and thus $\dom{\Lip_{\mathcal{I},\mu}^\beta}$ is dense in $\sa{\A}$.

On the other hand, if $p\leq n \in \N$ and $a\in\sa{\A}$, then:
\begin{equation}\label{AF-lip-norms-thm-eq2}
\|\mathds{E}_n(a)-\mathds{E}_p(\mathds{E}_n(a)) \|_\A = \|\mathds{E}_n(a-\mathds{E}_p(a))\|_\A \leq \|a-\mathds{E}_p(a)\|_\A\text{.}
\end{equation}
Thus, by Expressions (\ref{AF-lip-norms-thm-eq1}) and (\ref{AF-lip-norms-thm-eq2}), for all  $a\in\sa{\A}$,   
\begin{equation}\label{Lip-contr}
\Lip_{\mathcal{I},\mu}^\beta(\mathds{E}_n(a)) \leq \Lip_{\mathcal{I},\mu}^\beta(a).
\end{equation}

Last, let $\varepsilon > 0$. There exists $N\in\N$ such that for all $n\geq N$ we have $\beta(n)<\frac{\varepsilon}{2}$. Let:
\begin{equation*}
\alg{B}_N = \left\{ a\in\sa{\A_N} : \Lip_{\mathcal{I},\mu}^\beta(\indmor{\alpha}{N}(a)) \leq 1, \mu(a) = 0 \right\}\text{.}
\end{equation*}

Since $\mathds{E}_0 = \mu(\cdot)\unit_\A$, we conclude:
\begin{equation*}
\alg{B}_N \subseteq \{ a\in \sa{\A_N} : \|a\|_\A\leq \beta(0) \}\text{,}
\end{equation*}
and since a closed ball in $\sa{\A_N}$ is compact as $\A_N$ is finite dimensional, we conclude that $\alg{B}_N$ is totally bounded. Let $\alg{F}_N$ be a $\frac{\varepsilon}{2}$-dense subset of $\alg{B}_N$.

Let now $a\in\sa{\A}$ with $\mu(a) = 0$ and $\Lip_{\mathcal{I},\mu}^\beta(a)\leq 1$. By definition of $\Lip_{\mathcal{I},\mu}^\beta$  we have $\|a-\mathds{E}_N(a)\|_\A \leq \beta(N) < \frac{\varepsilon}{2}$. Moreover, there exists $a'\in\alg{F}_N$ such that $\|\mathds{E}_N(a)-a'\|_\A \leq \frac{\varepsilon}{2}$. Thus:
\begin{equation*}
\|a - a'\|_\A \leq \varepsilon \text{.}
\end{equation*}
Thus:
\begin{equation*}
\left\{ a\in\sa{\A} : \Lip_{\mathcal{I},\mu}^\beta(a)\leq 1, \mu(a) = 0 \right\}
\end{equation*}
is totally bounded. Thus $\Lip_{\mathcal{I},\mu}^\beta$ is a Lip-norm on $\A$.

We conclude with the observation that as the pointwise supremum of continuous functions, $\Lip_{\mathcal{I},\mu}^\beta$ is lower semi-continuous.

\begin{step}
If $n \in \N$, then $(\A_n , \Lip_{\mathcal{I},\mu}^\beta\circ\indmor{\alpha}{n})$ is a $(2,0)$-quasi-Leibniz quantum compact metric space and $\qpropinquity{} \left(\left(\A_n , \Lip_{\mathcal{I},\mu}^\beta\circ\indmor{\alpha}{n}\right), \left(\A , \Lip_{\mathcal{I},\mu}^\beta\right)\right) \leq \beta(n) .$
\end{step}

The restriction of $\Lip_{\mathcal{I},\mu}^\beta$ to $\indmor{\alpha}{n}(\A_n)$ is a $(2,0)$-quasi-Leibniz lower semi-continuous Lip-norm on $\indmor{\alpha}{n}(\A_n)$ for all $n\in\N$. 

Fix $n\in\N$. We now prove our estimate on $\qpropinquity{}\left(\left(\A_n,\Lip_{\mathcal{I},\mu}^\beta\circ\indmor{\alpha}{n}\right),\left(\A,\Lip_{\mathcal{I},\mu}^\beta\right)\right)$. 

The spaces $(\A_n,\Lip_{\mathcal{I},\mu}^\beta\circ\indmor{\alpha}{n})$ and $(\indmor{\alpha}{n}(\A_n), \Lip_{\mathcal{I},\mu}^\beta)$ are isometrically isomorphic and thus at distance zero for $\qpropinquity{}$. Therefore:
\begin{equation*}
\qpropinquity{}\left(\left(\A,\Lip_{\mathcal{I},\mu}^\beta\right),\left(\A_n,\Lip_{\mathcal{I},\mu}^\beta\circ\indmor{\alpha}{n}\right)\right) = \qpropinquity{}\left(\left(\A,\Lip_{\mathcal{I},\mu}^\beta\right),\left(\indmor{\alpha}{n}(\A_n),\Lip_{\mathcal{I},\mu}^\beta\right)\right)\text{.}
\end{equation*}

Let $\mathrm{id} : \A\rightarrow\A$ be the identity and let $\iota_n : \indmor{\alpha}{n}(\A_n ) \rightarrow \A$ be the inclusion map.  The quadruple $\gamma=(\A,\unit_\A,\iota_{n},\mathrm{id})$ is a bridge from $\indmor{\alpha}{n}(\A_n)$ to $\A$ by Definition (\ref{bridge-def}). We note that by definition, the height of $\gamma$ is $0$ since the pivot of $\gamma$ is $\unit_\A$. Thus, the length of $\gamma$ is the reach of $\gamma$.

If $a \in \sa{\A}$ with $\Lip_{\mathcal{I},\mu}^\beta(a)\leq 1$, then:
\begin{equation*}
\Vert a - \mathds{E}_n (a) \Vert_\A \leq \beta(n)\text{.}
\end{equation*}
Since $\mathds{E}_n$ is positive, we thus have $\mathds{E}_n (a) \in \sa {\indmor{\alpha}{n}(\A_n)}$ .  By Equation (\ref{Lip-contr}):
\begin{equation*} 
\Lip_{\mathcal{I},\mu}^\beta\left(\mathds{E}_n (a)\right) \leq 1\text{.}
\end{equation*}
Since $\indmor{\alpha}{n}(\A_n)$ is contained in $\A$,  we conclude that the reach of $\gamma$ is no more than $\beta(n)$.

We thus conclude, by definition:
\begin{equation*}
\qpropinquity{} \left(\left(\indmor{\alpha}{n}(\A_n) , \Lip_{\mathcal{I},\mu}^\beta\right), \left(\A , \Lip_{\mathcal{I},\mu}^\beta\right)\right) \leq \beta(n)\text{.}
\end{equation*}

As $(\beta(n))_{n\in\N}$ converges to $0$, we conclude that:
\begin{equation*}
\lim_{n\rightarrow \infty} \qpropinquity{}\left(\left(\A_n,\Lip_{\mathcal{I},\mu}^\beta\circ\indmor{\alpha}{n}\right),\left(\A,\Lip_{\mathcal{I},\mu}^\beta\right)\right) = 0\text{,}
\end{equation*}
and thus our theorem is proven.
\end{proof}

\begin{remark}
We may employ similar techniques as used in the proof of Theorem (\ref{AF-lip-norms-thm}) to show that AF algebras, equipped with the Lip-norms defined from spectral triples in \cite{Antonescu04}, are limits of finite dimensional C*-subalgebras. We shall see in this paper, however, that the Lip-norms we introduce in Theorem (\ref{AF-lip-norms-thm}) provide a very natural framework to study the quantum metric properties of AF algebras.
\end{remark}

Theorem (\ref{AF-lip-norms-thm}) provides infinitely many Lip-norms on any given unital AF-algebra $\A$, parametrized by a choice of an inductive sequence converging to $\A$ and a sequence of positive entries which converge to $0$. A natural choice of a Lip-norm for a given AF algebra, which will occupy a central role in our current work, is described in the following notation.

\begin{notation}
Let $\mathcal{I} = (\A_n,\alpha_n)_{n\in\N}$ be a unital inductive sequence of finite dimensional algebras whose inductive limit $\A = \varinjlim (\A_n,\alpha_n)_{n\in\N}$ has a faithful tracial state $\mu$. Assume that $\A$ is infinite dimensional. Let $k\in \N$, $k>0$ and $\beta = \left(\frac{1}{\dim(\A_n)^k} \right)_{n\in\N}$. We note that $\lim_{\infty} \beta = 0$. We denote the Lip-norm $\Lip_{\mathcal{I},\mu}^\beta$ constructed in Theorem (\ref{AF-lip-norms-thm}) by $\Lip_{\mathcal{I},\mu}^k$. If $k=1$, then we simply write $\Lip_{\mathcal{I},\mu}$ for $\Lip_{\mathcal{I},\mu}^1$.  
\end{notation}

Our purpose is the study of various classes of AF algebras, equipped with Lip-norms constructed in Theorem (\ref{AF-lip-norms-thm}). The following notation will prove useful.

\begin{notation}\label{AF-class-notation}
The class of all {\Qqcms{2,0}s} constructed in Theorem (\ref{AF-lip-norms-thm}) is denoted by $\mathcal{AF}$. We shall endow $\mathcal{AF}$ with the topology induced by the quantum propinquity $\qpropinquity{}$. 

Furthermore, for any $k\in (0,\infty)$, let:
\begin{equation*}
\mathcal{AF}^k := \left\{ (\A, \Lip_\A) \in \mathcal{AF} \middle\vert \begin{array}{l} \exists \mathcal{I} \in \text{\calligra Inductive-f-d} \quad \A = \varinjlim \mathcal{I} \\ \exists\mu\text{ faithful trace on $\A$ such that } \Lip_\A = \Lip_{\mathcal{I},\mu}^k \\ \A\text{ is infinite dimensional} \end{array} \right\}
\end{equation*}
where {\calligra Inductive-f-d } is the class of all unital inductive sequences of finite dimensional C*-algebras whose limit has at least one faithful tracial state.
\end{notation}  

A first corollary of Theorem (\ref{AF-lip-norms-thm}) concerns some basic geometric properties of the class $\mathcal{AF}^k$:

\begin{notation}
We denote the diameter of any metric space $(X,\mathrm{d})$ by $\diam{X}{\mathrm{d}}$.

For any quantum compact metric space $(\A,\Lip)$, we denote $\diam{\StateSpace(\A)}{\Kantorovich{\Lip}}$ by $\qdiam{\A}{\Lip}$.
\end{notation}

\begin{corollary}\label{fin-diam-cor}
Let $\mathcal{I}, \mathcal{J} \in \text{\calligra Inductive-f-d}$ and $\beta, \beta'$ be two sequences of strictly positive real numbers, converging to $0$. Let $\mu$, $\nu$ be faithful tracial states, respectively, on $\varinjlim\mathcal{I}$ and $\varinjlim\mathcal{J}$. Then:
\begin{equation*}
\qdiam{\varinjlim \mathcal{I}}{\Lip_{\mathcal{I},\mu}^\beta} \leq 2\beta(0)
\end{equation*}
and:
\begin{equation*}
\qpropinquity{}\left(\left(\varinjlim \mathcal{I},\Lip_{\mathcal{I},\mu}^\beta\right), \left(\varinjlim \mathcal{J},\Lip_{\mathcal{J},\nu}^{\beta'}\right)\right) \leq \max\{\beta(0),\beta'(0)\}\text{.}
\end{equation*}
In particular, for all $k\in (0,\infty)$:
\begin{equation*}
\diam{\mathcal{AF}^k}{\mathsf{\Lambda}} \leq 1 .
\end{equation*}
\end{corollary}

\begin{proof}
Let $\A = \varinjlim \mathcal{I}$ and $\B = \varinjlim \mathcal{J}$.

Let $a\in\sa{\A}$ with $\Lip_{\mathcal{I},\mu}^\beta(a)\leq 1$. Then $\|a-\mu(a)\|_\A\leq \beta(0)$. Thus for any $\varphi, \psi \in \StateSpace(\A)$, we have:
\begin{equation*}
|\varphi(a)-\psi(a)| = |\varphi(a-\mu(a)\unit_\A) - \psi(a-\mu(a)\unit_\A)| \leq 2\beta(0)\text{.}
\end{equation*}

Now, let $\D = \A * \B$ be the free product amalgamated over $\C\unit_\A$ and $\C\unit_\B$. Let $\pi : \A \hookrightarrow \D$ and $\rho : \B \hookrightarrow \D$ be the canonical unital *-monomorphism. The quadruple $\gamma = (\D, \unit_\D, \pi, \rho)$ is a bridge from $\A$ to $\B$. 

Let $a\in \sa{\A}$ with $\Lip_{\mathcal{I},\mu}^\beta (a) \leq 1$. Then:
\begin{equation*}
\|\pi(a) \unit_\D - \unit_\D\rho(\mu(a)\unit_\B)\|_\D = \|a-\mu(a)\unit_\A\|_\A \leq \beta(0)\text{.}
\end{equation*}
The result is symmetric in $\A$ and $\B$. Thus the reach of $\gamma$ is no more than $\max\{\beta(0),\beta'(0)\}$. As the height of $\gamma$ is zero, we have proven that:
\begin{equation*}
\qpropinquity{}\left(\left(\varinjlim \mathcal{I},\Lip_{\mathcal{I},\mu}^\beta\right), \left(\varinjlim \mathcal{J},\Lip_{\mathcal{J},\nu}^{\beta'}\right)\right) \leq \max\{\beta(0),\beta'(0)\}\text{.}
\end{equation*}
Note that this last estimate is slightly better than what we would obtain with \cite[Proposition 4.6]{Latremoliere13}.

We conclude our proof noting that if $(\A,\Lip_{\mathcal{I}})\in\mathcal{AF}^k$ then $\beta(0) = 1$.
\end{proof}

\section{The geometry of the class of UHF Algebras for $\qpropinquity{}$}\label{uhf}

Our purpose for this section is to study the topology of the class of uniformly hyperfinite algebras equipped with the Lip-norms from Theorem (\ref{AF-lip-norms-thm}). We begin this section with an explicit computation, in this context, for the conditional expectations involved in our construction in Theorem (\ref{AF-lip-norms-thm}). We then establish our main result for this section, by constructing a continuous surjection from the Baire space to the subclass of $\mathcal{AF}^k$ consisting of UHF algebras.

\begin{notation}
For all $k\in (0,\infty)$, we let $\mathcal{UHF}^k$ be the subclass of $\mathcal{AF}^k$ of {\Qqcms{(2,0)}s} of the form $(\A,\Lip)$ with $\A$ a UHF algebra.
\end{notation}

\subsection{An expression for conditional expectations}\label{exp-formula-sec}

\begin{notation}
For all $d \in \N$, we denote the full matrix algebra of $d\times d$ matrices over $\C$ by $\alg{M}(d)$.
\end{notation}

Let $\B=\oplus_{j=1}^N \alg{M}(n(j))$ for some $N \in\N$ and $n(1),\ldots,n(N) \in \N\setminus\{0\}$. For each $k\in \{1,\ldots,N\}$ and for each $j,m \in \{1,\ldots,n(k)\}$, we denote the matrix $((\delta_{u,v}^{j,m}))_{u,v  =1,\ldots, n(k)}$ by $e_{k,j,m}$, where we used the Kronecker symbol:
\begin{equation*}
\delta_a^b = \begin{cases}
1 \text{ if $a=b$,}\\
0 \text{ otherwise.}
\end{cases}
\end{equation*} 

We note that for all $j,m,j',m' \in \{1,\ldots,n(k)\}$ we have:
\begin{equation*}
\mathrm{tr}\left( e_{k,j,m}^\ast e_{k,j',m'} \right) = \begin{cases}
\frac{1}{n(k)} \text{ if $j=j'$ and $m=m'$,}\\
0 \text{ otherwise}
\end{cases}
\end{equation*}
when $\mathrm{tr}$ is the unique tracial state of $\alg{M}(n(k))$.

Now, let $\mu$ be a faithful tracial state on $\B$. Then $\mu$ is a convex combination with positive coefficients of the unique tracial states on $\alg{M}(n(0)),\ldots,\alg{M}(n(N))$. We thus deduce that:
\begin{equation*}
\left\{ e_{k,j,m} : k\in\{1,\ldots,N\}, j,m\in \{1,\ldots,n(k)\} \right\}
\end{equation*}
is an orthogonal basis of $L^2(\B,\mu)$.

Let us further assume that we are given a unital *-monomorphism $\alpha : \B \hookrightarrow \A$ into a unital C*-algebra $\A$ with a faithful tracial state. The restriction of $\mu$ to $\alpha(\B)$ is thus a faithful tracial state on $\alpha(\B)$. We will use the notations of the proof of Theorem (\ref{AF-lip-norms-thm}): let $\pi$ be the GNS representation of $\A$ defined by $\mu$ on the Hilbert space $L^2(\A,\mu)$ and let $\xi : a\in\A\rightarrow a\in L^2(\A,\mu)$.

We then can regard $L^2(\alpha(\B),\mu)$ as a subspace of $L^2(\A,\mu)$ (as noted in the proof of Theorem (\ref{AF-lip-norms-thm}), $L^2(\alpha(\B),\mu)$ is $\alpha(\B)$, endowed with the Hermitian norm from the inner product defined by $\mu$). Let $P$ be the projection of $L^2(\A,\mu)$ on $L^2(\alpha(\B),\mu)$. Then for all $a\in \A$, we have:
\begin{equation} \label{cond-exp-eq}
P\xi(a) = \sum_{k=1}^N \sum_{j=1}^{n(k)}\sum_{m=1}^{n(k)} \frac{\mu(\alpha(e_{k,j,m}^\ast) a)}{\mu(\alpha(e_{k,j,m}^\ast e_{k,j,m}))} \alpha(e_{k,j,m})\text{.}
\end{equation}

We also note that, if $\CondExp{\cdot}{\alpha(\B)}$ is the conditional expectation of $\A$ onto $\alpha(\B)$ which preserves $\mu$ constructed from the Jones' projection $P$ as in Theorem (\ref{AF-lip-norms-thm}),  then $\xi(\CondExp{a}{\alpha(\B)}) = P\xi(a)$ for all $a\in\A$.

\subsection{A H{\"o}lder surjection from the Baire Space onto $\mathcal{UHF}^k$}

A uniform, hyperfinite (UHF) algebra is a particular type of AF algebra obtained as the limit of unital, simple finite dimensional C*-algebras. UHF algebras were classified by Glimm \cite{Glimm60} and, as AF algebras, they are also classified by their Elliott invariant \cite{Davidson}. UHF algebras are always unital simple AF algebras, and thus they admit a faithful tracial state. Moreover, the tracial state of a UHF algebra $\A$ is unique, as is seen by noting that it must restrict to the unique tracial state on the full matrix subalgebras of $\A$ whose union is dense in $\A$.

Up to unitary conjugation, a unital *-monomorphism $\alpha : \B\rightarrow\A$ between two unital simple finite dimensional C*-algebras, i.e. two nonzero full matrix algebras $\A$ and $\B$, exists if and only if $\dim\A = k^2\dim\B$ for $k\in\N$, and $\alpha$ must be of the form:
\begin{equation}\label{mul-morphism-eq}
A \in \B \longmapsto \begin{pmatrix}
A & & \\
& \ddots & \\
& & A
\end{pmatrix}\in \A\text{.}
\end{equation}

It is thus sufficient, in order to characterize a unital inductive sequence of full matrix algebras, to give a sequence of positive integers:

\begin{definition}
Let $\mathcal{I} = (\A_n,\alpha_n)_{n\in\N}$ be a unital inductive sequence of unital, simple finite dimensional C*-algebras, with $\A_0 = \C$.

The \emph{multiplicity sequence} of $\mathcal{I}$ is the sequence $\left(\sqrt{\frac{\dim\A_{n+1}}{\dim\A_n}} \right)_{n\in\N}$ of positive integers.
\end{definition}

A multiplicity sequence is any sequence in $\N\setminus\{0\}$. A UHF algebra is always obtained as the limit of an inductive sequence in the following class:
\begin{notation}
Let \text{\calligra StrictFullInductive} be the set of all unital inductive sequences of full matrix algebras whose multiplicity sequence lies in $(\N\setminus\{0,1\})^\N$ and which starts with $\C$.
\end{notation}

UHF algebras have a unique tracial state, which is faithful since UHF algebras are simple. We make a simple observation relating multiplicity sequences and tracial states of the associated UHF algebras, which will be important for the main result of this section.

\begin{lemma}\label{uhf-trace-lemma}
Let $\mathcal{I} = (\A_n,\alpha_n)_{n\in\N}$ in \text{\calligra StrictFullInductive}. Let $\A = \varinjlim\mathcal{I}$ and let $\mu_\A$ be the unique tracial state of $\A$. Let $\vartheta$ be the multiplicity sequence of $\mathcal{I}$.
\begin{enumerate}
\item If $a\in \A_n$, then:
\begin{equation*}
\mu_\A(\indmor{\alpha}{n}(a)) = \frac{1}{\prod_{j=0}^{n-1} \vartheta(j)} \mathrm{Tr}(a)
\end{equation*}
where $\mathrm{Tr}$ is the unique trace on $\A_n$ which maps the identity to $\dim\A_n$.
\item Let $\mathcal{J} = (\B_n,\alpha'_n)_{n\in\N}$ in \text{\calligra StrictFullInductive} and set $\B = \varinjlim\mathcal{J}$. Let $\mu_\B$ the unique tracial state of $\B$. If the multiplicity sequences of $\mathcal{I}$ and $\mathcal{J}$ agree up to some $N\in\N$, then for all $n\in \{0,\ldots,N\}$, we have $\A_n = \B_n$ and moreover, for all $a\in \A_n = \B_n$, we have:
\begin{equation*}
\mu_\A \circ \indmor{\alpha}{n} (a) = \mu_\B \circ \indmor{\alpha'}{n}(a)\text{.}
\end{equation*}
\end{enumerate}

\end{lemma}

\begin{proof}
Assertion (1) follows from the uniqueness of the tracial state on $\A_n$ for all $n\in\N$.

Assertion (2) follows directly from Assertion (1).
\end{proof}

The set of sequences $\BaireSpace$ of positive integers is thus a natural parameter space for the classes $\mathcal{UHF}^k$. Moreover, $\BaireSpace$ can be endowed with a natural topology, and we thus can investigate the continuity of maps from the Baire space to $\left(\mathcal{UHF}^k,\qpropinquity{}\right)$.

\begin{definition}\label{Baire-Space-def}
The \emph{Baire space} $\BaireSpace$ is the set $(\N\setminus\{0\})^\N$ endowed with the metric $\mathsf{d}$ defined, for any two $(x(n))_{n\in\N}$, $(y(n))_{n\in\N}$ in $\BaireSpace$, by:
\begin{equation*}
\mathsf{d}\left((x(n))_{n\in\N}, (y(n))_{n\in\N}\right) = \begin{cases}
0 \text{ if $x(n) = y(n)$ for all $n\in\N$},\\
2^{-\min\left\{ n \in \N : x(n) \not= y(n) \right\}} \text{ otherwise}\text{.}
\end{cases}
\end{equation*}
\end{definition}

\begin{remark}
We note that it is common, in the literature on descriptive set theory, to employ the metric defined on $\BaireSpace$ by setting on $(x(n))_{n\in\N}, (y(n))_{n\in\N} \in \BaireSpace$:
\begin{equation*}
\mathsf{d'}\left((x(n))_{n\in\N}, (y(n))_{n\in\N}\right) = \begin{cases}
0 \text{ if $x(n) = y(n)$ for all $n\in\N$},\\
\frac{1}{1+\min\left\{ n \in \N : x(n) \not= y(n) \right\}} \text{ otherwise}\text{.}
\end{cases}
\end{equation*}
It is however easy to check that $\mathsf{d}$ and $\mathsf{d}'$ are topologically, and in fact uniformly equivalent as metrics. Our choice will make certain statements in our paper more natural.
\end{remark}

We now prove the result of this section: there exists a natural continuous surjection from the Baire space $\BaireSpace$ onto $\mathcal{UHF}^k$ for all $k \in (0,\infty)$. We recall:

\begin{definition}
A function $f : X\rightarrow Y$ between two metric spaces $(X,\mathsf{d}_X)$ and $(Y,\mathsf{d}_Y)$ is $(c,r)$-H{\"older}, for some $c \geq 0$ and $r > 0$, when:
\begin{equation*}
\mathsf{d}_Y(f(x),f(y)) \leq c \mathsf{d}_X(x,y)^r 
\end{equation*}
for all $x,y \in X$.
\end{definition}

\setcounter{step}{0}
\begin{theorem}\label{uhf-thm} 
For any $\beta = (\beta(n))_{n\in\N} \in \BaireSpace$, we define the sequence $\boxtimes\beta$ by:
\begin{equation*}
\boxtimes\beta = n \in \N \longmapsto \begin{cases}
1 \text{ if $n = 0$,}\\
\prod_{j=0}^{n-1} (\beta(j)+1) \text{ otherwise}.
\end{cases}
\end{equation*}

We then define, for all $\beta\in\BaireSpace$, the unital inductive sequence:
\begin{equation*}
\mathcal{I}(\beta) = \left(\alg{M}\left(\boxtimes\beta(n)\right), \alpha_n\right)_{n\in\N}
\end{equation*}
where $\alg{M}(d)$ is the algebra of $d\times d$ matrices and for all $n\in\N$, the unital *-monomorphism $\alpha_n$ is of the form given in Expression (\ref{mul-morphism-eq}).

The map $\alg{u}$ from $\BaireSpace$ to the class of UHF algebras is now defined by:
\begin{equation*}
(\beta(n))_{n\in\N}\in\BaireSpace  \longmapsto \alg{u}((\beta(n))_{n\in\N}) = \varinjlim\mathcal{I}(\beta)\text{.}
\end{equation*}

Let $k\in (0,\infty)$ and $\beta\in\BaireSpace$. Let $\Lip_\beta^k$ be the Lip-norm $\Lip_{\mathcal{I}(\beta),\mu}^\vartheta$ on $\alg{u}(\beta)$ given by Theorem (\ref{AF-lip-norms-thm}), the sequence $\vartheta : n\in\N\mapsto \boxtimes\beta(n)^k$ and the unique faithful trace $\mu$ on $\alg{u}(\beta)$.

The {\Qqcms{(2,0)}} $\left(\alg{u}(\beta), \Lip_{\beta}^k\right)$ will be denoted simply by $\uhf{\beta,k}$.

For all $k\in(0,\infty)$, the map:
\begin{equation*}
\uhf{\cdot,k} : \BaireSpace \longrightarrow \mathcal{UHF}^k
\end{equation*}
is a $(2,k)$-H{\"o}lder surjection.
\end{theorem}

\begin{proof}
We fix $k \in (0,\infty)$. Let $\beta \in \BaireSpace$ and write $\mathcal{I}(\beta) = (\A_n,\alpha_n)_{n\in\N}$. Note that $\A_n = \alg{M}(\boxtimes\beta(n))$ for all $n\in\N$. Moreover, we denote $\uhf{\beta,k}$ by $(\A,\Lip_\A)$.

We begin with a uniform estimate on the propinquity.

Fix $n\in\N$. By definition, $\boxtimes\beta(n) \geq 2^n$. By Theorem (\ref{AF-lip-norms-thm}), we conclude:
\begin{equation*}
\qpropinquity{}((\A,\Lip_\A), (\indmor{\alpha}{n}(\A_n),\Lip_\A)) \leq \boxtimes\beta(n)^{-k} \leq 2^{-nk}\text{.}
\end{equation*}

Now, $(\indmor{\alpha}{n}(\A_n),\Lip_\A)$ and $(\A_n,\Lip_\A\circ\indmor{\alpha}{n})$ are isometrically isomorphic, so:
\begin{equation}\label{uhf-thm-eq1}
\qpropinquity{}((\A,\Lip_\A), (\A_n,\Lip_\A\circ\indmor{\alpha}{n})) \leq 2^{-nk}\text{.}
\end{equation}

Let now $\eta \in \BaireSpace$ and write $\mathcal{I}(\eta) = (\B_n,\alpha_n')_{n\in\N}$. Note that $\B_n = \alg{M}(\boxtimes\eta(n))$ for all $n\in\N$. Moreover, we denote $\uhf{\eta,k}$ by $(\B,\Lip_\B)$.

Let $N = -\log_2 \mathsf{d}(\beta,\eta)  \in \N $. 

If $N =0$, then the best estimate at our disposal is given by Corollary (\ref{fin-diam-cor}), and we conclude:
\begin{equation*}
\qpropinquity{}((\A,\Lip_\A),(\B,\Lip_\B)) \leq \max\{\boxtimes\beta(0),\boxtimes\eta(0)\} = 1 = \mathsf{d}(\eta,\beta)\text{.}
\end{equation*}

Assume now that $N \geq 1$. By definition, $\boxtimes\beta(j) =\boxtimes \eta(j)$ for all $j\in \{0,\ldots,N\}$. By Lemma (\ref{uhf-trace-lemma}), we note that $\A_N = \B_N = \alg{M}(\boxtimes\beta(N))$, and moreover:
\begin{equation*}
\mu_\A\circ\indmor{\alpha}{j} = \mu_\B\circ\indmor{\alpha'}{j}
\end{equation*}
for all $j\in \{0,\ldots,N\}$.

We now employ the notations of Section (\ref{exp-formula-sec}). For all $j\in \{0,\ldots,N\}$, we thus fix the canonical set $\{e_{k,m} \in \alg{M}(\boxtimes\beta(j)) : k, m \in I_j \}$ of $\alg{M}(\boxtimes\beta(j))$, where:
\begin{equation*}
I_j = \left\{ (k,m) \in \N^2 : 1\leq k,m \leq \boxtimes\beta(j)\right\}\text{.}
\end{equation*} 

Next, for all $j \in \{0, \ldots, N\}$, we have that $(\A_j,\alpha_j)=(\B_j, {\alpha'}_j )$. Therefore, if $j \in \{0, \ldots, N-1\}$,  then $\alpha_{j,N-1} = \alpha_{N-1} \circ \cdots \circ \alpha_j = {\alpha'}_{N-1} \circ \cdots \circ {\alpha'}_j = {\alpha'}_{j,N-1}$.  Also, by definition of the canonical maps $\indmor{\alpha}{n}$ and definition of inductive limit, we have that if $c \in \A_j$, then $\indmor{\alpha}{j}(c)=\indmor{\alpha}{N}(\alpha_{j,N-1} (c))=\indmor{\alpha}{N}({\alpha'}_{j,N-1} (c))$ for $j \in \{0, \ldots , N-1\}.$ Thus, from Expression (\ref{cond-exp-eq}) for all $a\in \alg{M}(\boxtimes \beta(N)), j \in \{0, \ldots, N-1\}$ we note:

\begin{equation}\label{induct-trace}
\begin{split}
&\left\| \indmor{\alpha}{N}(a) - \CondExp{\indmor{\alpha}{N}(a)}{\indmor{\alpha}{j}(\A_j)} \right\|_\A \\
& = \left\|\indmor{\alpha}{N}(a) - \sum_{l \in I_j}\frac{\mu_\A\left(\indmor{\alpha}{j}(e_l^\ast) \indmor{\alpha}{N}(a)\right)}{\mu_\A\left(\indmor{\alpha}{j}(e_l^\ast e_l)\right)}\indmor{\alpha}{j}(e_l)\right\|_\A \\
& =\left\|\indmor{\alpha}{N}(a) - \sum_{l \in I_j}\frac{\mu_\A\left(\indmor{\alpha}{N}(\alpha_{j,N-1}(e_l^\ast)) \indmor{\alpha}{N}(a)\right)}{\mu_\A\left(\indmor{\alpha}{N}(\alpha_{j,N-1}(e_l^\ast e_l))\right)}\indmor{\alpha}{N}(\alpha_{j,N-1}(e_l))\right\|_\A\\
&=\left\|a - \sum_{l \in I_j}\frac{\mu_\A\left(\indmor{\alpha}{N}(\alpha_{j,N-1}(e_l^\ast) a)\right)}{\mu_\A\left(\indmor{\alpha}{N}(\alpha_{j,N-1}(e_l^\ast e_l))\right)}\alpha_{j,N-1}(e_l)\right\|_{\alg{M}(\boxtimes\beta(N))} \\
&=\left\|a - \sum_{l \in I_j}\frac{\mu_\A \circ\indmor{\alpha}{N}\left({\alpha'}_{j,N-1}(e_l^\ast) a\right)}{\mu_\A \circ \indmor{\alpha}{N} \left({\alpha'}_{j,N-1}(e_l^\ast e_l)\right)}{\alpha'}_{j,N-1}(e_l)\right\|_{\alg{M}(\boxtimes\beta(N))}\\
&=\left\|a - \sum_{l \in I_j}\frac{\mu_\B \circ\indmor{\alpha'}{N}\left({\alpha'}_{j,N-1}(e_l^\ast) a\right)}{\mu_\B \circ \indmor{\alpha'}{N} \left({\alpha'}_{j,N-1}(e_l^\ast e_l)\right)}{\alpha'}_{j,N-1}(e_l)\right\|_{\alg{M}(\boxtimes\beta(N))}\\
&=\left\|\indmor{\alpha'}{N}(a) - \sum_{l \in I_j}\frac{\mu_\B\left(\indmor{\alpha'}{N}({\alpha'}_{j,N-1}(e_l^\ast) a)\right)}{\mu_\B\left(\indmor{\alpha'}{N}({\alpha'}_{j,N-1}(e_l^\ast e_l))\right)}\indmor{\alpha'}{N}({\alpha'}_{j,N-1}(e_l))\right\|_\B \\
& = \left\|\indmor{\alpha'}{N}(a) - \sum_{l \in I_j}\frac{\mu_\B\left(\indmor{\alpha'}{j}(e_l^\ast) \indmor{\alpha'}{N}(a)\right)}{\mu_\A\left(\indmor{\alpha'}{j}(e_l^\ast e_l)\right)}\indmor{\alpha'}{j}(e_l)\right\|_\B \\
&= \left\| \indmor{\alpha'}{N}(a) - \CondExp{\indmor{\alpha'}{N}(a)}{\indmor{\alpha'}{j}(\B_j)} \right\|_\B\text{.}
\end{split}
\end{equation}
If $j\geq N$, then $\CondExp{\indmor{\alpha}{N}(a)}{\indmor{\alpha}{j}(\A_j)}={\indmor{\alpha}{N}(a)}$ and $\CondExp{\indmor{\alpha'}{N}(a)}{\indmor{\alpha'}{j}(\B_j)}={\indmor{\alpha'}{N}(a)}$  by definition of conditional expectation.
Consequently, by definition:
\begin{equation*}
\Lip_\A\circ\indmor{\alpha}{N} = \Lip_\B\circ\indmor{\alpha'}{N}\text{,}
\end{equation*}
so:
\begin{equation}\label{uhf-thm-eq2}
\qpropinquity{}((\A_N,\Lip_\A\circ\indmor{\alpha}{N}), (\B_N, \Lip_\B\circ\indmor{\alpha'}{N})) = 0 \text{.}
\end{equation}

Hence, by the triangle inequality applied to Inequalities (\ref{uhf-thm-eq1}) and (\ref{uhf-thm-eq2}):
\begin{equation*}
\qpropinquity{}(\uhf{\beta,k}, \uhf{\eta,k}) \leq \frac{2}{2^{Nk}} \leq 2\mathsf{d}(\beta,\eta)^k \text{.}
\end{equation*}

Last, we show that the map $\uhf{\cdot,k}$ is a surjection. If $\alg{U}$ is a UHF algebra, then there exists an inductive sequence $\mathcal{I} = (\A_n,\alpha_n)_{n\in\N}$ of full matrix algebras whose limit is $\alg{U}$ and such that $\A_0 = \C$, while the multiplicity sequence $\beta$ of $\mathcal{I}$ is in $\N\setminus\{0,1\}$. Thus $\alg{u}((\beta(n)-1)_{n\in\N}) = \alg{U}$. Moreover, any Lip-norm $\Lip$ on $\alg{U}$ such that $(\alg{U},\Lip) \in \mathcal{UHF}^k$ can be obtained, by definition, from such a multiplicity sequence.

This concludes our theorem.
\end{proof}

\begin{remark}
Inequality (\ref{uhf-thm-eq1}) is sharp, as it becomes an inequality for the sequence $\beta = (1,1,1,\ldots) \in \BaireSpace$, and we note that the UHF algebra $\alg{u}(\beta)$ is the CAR algebra.
\end{remark}

\begin{remark}
Since $\mathsf{d}$ is an ultrametric on $\BaireSpace$, we conclude that $\mathsf{d}^k$ is a topologically equivalent ultrametric on $\BaireSpace$ as well. Hence, we could reformulate the conclusion of Theorem (\ref{uhf-thm}) by stating that $\uhf{\cdot,k}$ is $2$-Lipschitz for $\mathsf{d}^k$.
\end{remark}

\section{The geometry of the class of Effros-Shen AF algebras $\af{\theta}$ for $\qpropinquity{}$}\label{af-theta}

The original classification of irrational rotation algebras, due to Pimsner and Voi\-cu\-le\-scu \cite{PimVoi80a}, relied on certain embeddings into the AF algebras constructed from continued fraction expansions by Effros and Shen \cite{Effros80b}. In \cite{Latremoliere13c}, the second author proved that the irrational rotational algebras vary continuously in quantum propinquity with respect to their irrational parameter. It is natural to wonder whether the AF algebras constructed by Pimsner and Voiculescu vary continuously with respect to the quantum propinquity if parametrized by the irrational numbers at the root of their construction. We shall provide a positive answer to this problem in this section.

\subsection{Construction of $\af{\theta}$ for all $\theta \in (0,1)\setminus\Q$}
 
We begin by recalling the construction of the AF C*-algebras $\alg{AF}_\theta$ constructed in \cite{Effros80b} for any irrational $\theta$ in $(0,1)$. For any $\theta \in (0,1)\setminus\Q$, let $(r_j)_{j\in\N}$ be the unique sequence in $\N$ such that:
\begin{equation}\label{continued-fraction-eq}
\theta = \lim_{n\rightarrow\infty} r_0 + \cfrac{1}{r_1 + \cfrac{1}{r_2 + \cfrac{1}{r_3 +\cfrac{1}{\ddots+\cfrac{1}{r_n}}}}}\text{.}
\end{equation}
The sequence $(r_j)_{j\in\N}$ is called the continued fraction expansion of $\theta$, and we will simply denote it by writing $\theta = [r_0 , r_1 , r_2, \ldots ] = [r_j]_{j\in\N}$. We note that $r_0 = 0$ (since $\theta\in(0,1)$) and $r_n \in \N\setminus\{0\}$ for $n \geq 1 $.

We fix $\theta \in (0,1)\setminus\Q$, and let $\theta = [r_j]_{j\in\N}$ be its continued fraction decomposition. We then obtain a sequence $\left(\frac{p_n^\theta}{q_n^\theta}\right)_{n\in\N}$ with $p_n^\theta \in \N$ and $q_n^\theta \in \N\setminus\{0\}$ by setting:
\begin{equation}\label{pq-rel-eq}
\begin{cases}
\begin{pmatrix}p_1^\theta & q_1^\theta \\ p_0^\theta & q_0^\theta \end{pmatrix} = \begin{pmatrix}r_0r_1+1 & r_1 \\ r_0 & 1 \end{pmatrix}\\
\begin{pmatrix}p_{n+1}^\theta & q_{n+1}^\theta \\ p_n^\theta & q_n^\theta \end{pmatrix} = \begin{pmatrix}r_{n+1} & 1 \\ 1 & 0 \end{pmatrix}\begin{pmatrix}p_n^\theta & q_n^\theta \\ p_{n-1}^\theta & q_{n-1}^\theta \end{pmatrix}\text{ for all $n\in \N\setminus\{0\}$.}
\end{cases}
\end{equation}
We then note that $\left(\frac{p_n^\theta}{q_n^\theta}\right)_{n\in\N}$ converges to $\theta$.

Expression (\ref{pq-rel-eq}) contains the crux for the construction of the Effros-Shen AF algebras. 

\begin{notation}
Throughout this paper, we shall employ the notation $x\oplus y \in X\oplus Y$ to mean that $x\in X$ and $y\in Y$ for any two vector spaces $X$ and $Y$ whenever no confusion may arise, as a slight yet convenient abuse of notation.
\end{notation}

\begin{notation}\label{af-theta-notation}
Let $\theta \in (0,1)\setminus\Q$ and $\theta = [r_j]_{j\in\N}$ be the continued fraction expansion of $\theta$. Let $(p_n^\theta)_{n\in\N}$ and $(q_n^\theta)_{n\in\N}$ be defined by Expression (\ref{pq-rel-eq}). We set $\af{\theta,0} = \C$ and, for all $n\in\N\setminus\{0\}$, we set:
\begin{equation*}
\af{\theta, n} = \alg{M}(q_{n}^\theta) \oplus \alg{M}(q_{n-1}^\theta) \text{,}
\end{equation*}
and:
\begin{equation*}
\alpha_{\theta,n} : a\oplus b \in \af{\theta,n} \longmapsto \begin{pmatrix}
a & & &  \\
  & \ddots & & \\
  &        & a & \\
  &        &   & b 
\end{pmatrix} \oplus a \in \af{\theta, n+1} \text{,}
\end{equation*}
where $a$ appears $r_{n+1}$ times on the diagonal of the right hand side matrix above. We also set $\alpha_0$ to be the unique unital *-morphism from $\C$ to $\af{\theta,1}$.

We thus define the Effros-Shen C*-algebra $\af{\theta}$, after \cite{Effros80b}:
\begin{equation*}
\af{\theta} = \varinjlim \left(\af{\theta,n}, \alpha_{\theta,n}\right)_{n\in\N}\text{.}
\end{equation*} 
\end{notation}

In \cite{PimVoi80a}, Pimsner and Voiculescu construct, for any $\theta\in(0,1)\setminus\Q$, a unital *-mono\-mor\-phism from the irrational rotation C*-algebra $\A_\theta$, i.e. the universal C*-algebra generated by two unitaries $U$ and $V$ subject to $UV = \exp(2i\pi\theta)VU$, into $\af{\theta}$. This was a crucial step in their classification of irrational rotation algebras and started a long and fascinating line of investigation about AF embeddings of various C*-algebras.

In order to apply our Theorem (\ref{AF-lip-norms-thm}), we need to find a faithful tracial state on $\af{\theta}$, for all $\theta\in (0,1)\setminus\Q$. This is the matter we address in our next subsection.

\subsection{The tracial state of $\af{\theta}$}

We shall prove that for all $\theta\in(0,1)\setminus\Q$, there exists a unique tracial state on $\af{\theta}$ which will be faithful as $\af{\theta}$ is simple (note that there must exists at least one tracial state on any unital simple AF algebra). The source of our tracial state will be the K-theory of $\af{\theta}$.

We refer to \cite[Section VI.3]{Davidson} for the computation of the Elliott invariant of $\af{\theta}$, which reads:
\begin{theorem}[\cite{Effros80b}]
Let $\theta\in(0,1)\setminus\Q$ and let $C_\theta= \{(x,y) \in \Z^2 : \theta x+ y \geq 0 \}$. Then $K_0(\af{\theta}) = \Z^2$ with positive cone $C_\theta$ and order unit $(0,1)$.   
Thus the only state of the ordered group $(K_0(\af{\theta}), C_\theta,(0,1))$ is given by the map:
\begin{equation*}
(x,y) \in \Z^2 \longmapsto \theta x + y\text{.}
\end{equation*}
Thus $\af{\theta}$ has a unique tracial state, denoted by $\sigma_\theta$.
\end{theorem}

\begin{notation}\label{af-theta-lip-norm-notation}
Let $\theta\in(0,1)\setminus \Q$ and $k \in (0,\infty)$. The Lip-norm $\Lip_\theta^k$ on $\af{\theta}$ is the lower semi-continuous, $(2,0)$-quasi Leibniz Lip-norm $\Lip_{\mathcal{I}(\theta),\sigma_\theta}^k$ defined in Notation (\ref{AF-class-notation}) based on Theorem (\ref{AF-lip-norms-thm}), where $\mathcal{I}(\theta) = (\af{\theta,n},\alpha_{\theta,n})_{n\in\N}$ as in Notation (\ref{af-theta-notation}).
\end{notation}

As Theorem (\ref{AF-lip-norms-thm}) provides Lip-norms based, in part, on the choice of a faithful tracial state, a more precise understanding of the unique faithful tracial state of $\af{\theta}$ is required. We summarize our observations in the following lemma.

\begin{lemma}\label{af-theta-trace-lemma}
Let $\theta \in (0,1)\setminus\Q$ and let $\sigma_\theta$ be the unique tracial state of $\af{\theta}$, and fix $n\in\N\setminus\{0\}$. Using Notation (\ref{af-theta-notation}), let:
\begin{equation*}
\sigma_{\theta,n} = \sigma_\theta\circ\indmor{\alpha_\theta}{n}\text{.}
\end{equation*}
Let $\mathrm{tr}_d$ be the unique tracial state on $\alg{M}(d)$ for any $d\in\N$. Then, if $(p^\theta_n)_{n\in\N}$ and $(q_n^\theta)_{n\in\N}$ are defined by Expression (\ref{pq-rel-eq}), then:
\begin{equation*}
\sigma_{\theta,n} : a\oplus b \in \af{\theta,n} \longmapsto t(\theta,n)\mathrm{tr}_{q_{n}^\theta}(a) + (1-t(\theta,n))\mathrm{tr}_{q_{n-1}^\theta}(b)\text{,}
\end{equation*}
where
\begin{equation*}
t(\theta,n) = (-1)^{n-1} q_n^\theta\left(\theta q_{n-1}^\theta - p_{n-1}^\theta\right) \in (0,1) \text{.}
\end{equation*}
\end{lemma}

\begin{proof}
The map $\sigma_{\theta,n}$ is a tracial state on $\af{\theta,n} = \alg{M}(q^\theta_n)\oplus\alg{M}(q^\theta_{n-1})$, and thus there exists $t(n,\theta) \in [0,1]$ such that for all $a\oplus b\in\af{\theta,n}$:
\begin{equation*}
\sigma_{\theta,n} (a\oplus b) =t(\theta,n)  \mathsf{tr}_{q^\theta_n}(a) +(1-t(\theta,n))\mathsf{tr}_{q^\theta_{n-1}}(b)\text{.}
\end{equation*}
Let $\sigma_\ast : K_0(\af{\theta})\rightarrow\R$ be the state induced by $\sigma_\theta$ on the $K_0$ group of $\af{\theta}$. We then have:
\begin{equation}\label{af-theta-trace-lemma-eq1}
\begin{split}
t(\theta,n) &= \sigma_{\theta,n}(\unit_{\alg{M}(q^\theta_{n})}\oplus 0) \\
&= \sigma_{\theta}\circ\indmor{\alpha_\theta}{n}(\unit_{\alg{M}(q^\theta_{n})}\oplus 0)\\
&= \sigma_\ast\circ K_0\left(\indmor{\alpha_\theta}{n}\right)\left(\begin{pmatrix}q^\theta_n \\ 0\end{pmatrix}\right)
\end{split}
\end{equation}
where $K_0\left(\indmor{\alpha_\theta}{n}\right)$ is the map from $K_0(\af{\theta,n}) = \Z^2$ to $K_0(\af{\theta}) = \Z^2$ induced by $\indmor{\alpha}{n}$. By construction, following \cite[Section VI.3]{Davidson}, we have:
\begin{equation*}
K_0\left(\indmor{\alpha_\theta}{n}\right) \begin{pmatrix} z1 \\ z2 \end{pmatrix} = (-1)^{n-1} \begin{pmatrix} q^\theta_{n-1} & -q^\theta_{n} \\ -p^\theta_{n-1} & p^\theta_n \end{pmatrix} \begin{pmatrix} z1 \\ z2 \end{pmatrix}
\end{equation*}
for all $(z1,z2) \in \Z^2$. Therefore:
\begin{equation*}
\begin{split}
t(\theta,n) &= (-1)^{n-1} \sigma_\ast \left(\begin{pmatrix}
q^\theta_{n-1} & -q^\theta_{n} \\ -p^\theta_{n-1} & p^\theta_n  
\end{pmatrix} 
\begin{pmatrix} q^\theta_{n} \\ 0 \end{pmatrix} \right)\\
&= (-1)^{n-1} \sigma_\ast\left(\begin{pmatrix}  q^\theta_{n-1} q^\theta_{n} \\ -p^\theta_{n-1} q^\theta_{n} \end{pmatrix}\right)\\
&= (-1)^{n-1} q^\theta_{n}\left(\theta q^\theta_{n-1} - p^\theta_{n-1}\right)\text{.} 
\end{split}
\end{equation*}

Since $\theta$ is irrational, $t(\theta,n) \not= 0$. Since $\unit_{\alg{M}(q^\theta_{n})}\oplus 0$ is positive in $\af{\theta,n}$ and less than $\unit_{\af{\theta,n}}$, we conclude $t(\theta,n) \in (0,1]$.

To prove that $t(\theta,n) < 1$, we may proceed following two different routes. Applying a similar computation as in Expression (\ref{af-theta-trace-lemma-eq1}), we get:
\begin{equation*}
\sigma_{\theta,n}\left(0\oplus \unit_{\alg{M}(q^\theta_{n-1})}\right) = (-1)^{n} q^\theta_{n-1} \left(\theta q^\theta_{n} - p^\theta_{n}\right)\text{,}
\end{equation*}
and again as $\theta$ is irrational, this quantity is nonzero. As $1 = \sigma_{\theta,n}\left(\unit_{\alg{M}(q^\theta_{n})} \oplus \unit_{\alg{M}(q^\theta_{n-1})}\right)$, our lemma would thus be proven.

Instead, we employ properties of continued fraction expansions and note that since $p^\theta_{n}q^\theta_{n-1} - p^\theta_{n-1}q^\theta_{n} = (-1)^{n-1}$:
\begin{equation*}
\begin{split}
1-t(\theta,n) &= 1 - (-1)^{n-1} q^\theta_{n}(\theta q^\theta_{n-1} - p^\theta_{n-1})\\
&= (-1)^{n-1} \left( (-1)^{n-1} - q^\theta_{n}(\theta q^\theta_{n-1} - p^\theta_{n-1} )\right) \\
&= (-1)^{n-1} \left( p^\theta_{n}q^\theta_{n-1} - p^\theta_{n-1} q^\theta_{n} - q^\theta_{n}(\theta q^\theta_{n-1} - p^\theta_{n-1}) \right)\\
&= (-1)^{n} \left( q^\theta_{n}(\theta q^\theta_{n-1}) -p^\theta_{n}q^\theta_{n-1} \right)\\
&= (-1)^{n} q^\theta_{n-1} \left(\theta q^\theta_{n} - p^\theta_{n}\right)\text{,}
\end{split}
\end{equation*}
which is nonzero as $\theta$ is irrational, and is less than one since $t(\theta,n) > 0$. This concludes our proof.
\end{proof}

\begin{remark}
We may also employ properties of continued fractions expansions to show that $t(\theta,n) > 0$ for all $n\in\N$. We shall use the notations of the proof of Lemma (\ref{af-theta-trace-lemma}). We have:
\begin{equation*}
\frac{p^\theta_{2n}}{q^\theta_{2n}} < \theta < \frac{p^\theta_{2n+1}}{q^\theta_{2n+1}}
\end{equation*}
and thus $\theta q^\theta_{2n} - p^\theta_{2n} > 0$ and $p^\theta_{2n+1}-\theta q^\theta_{2n+1} > 0$, which shows that $t(\theta,n) > 0$ for all $n\in\N$ (note that $q^\theta_n \in \N\setminus\{0\}$ for all $n\in\N$ since $\theta > 0$).
\end{remark}

We wish to employ Expression (\ref{cond-exp-eq}) and thus, we will find the following computation helpful:

\begin{lemma}\label{mu-norm-lemma}
Let $\theta\in (0,1)\setminus\Q$ and let $n\in\N\setminus\{0\}$. Let $\{e_{1,j,m} \in \af{\theta,n} : 1\leq j,m \leq q^\theta_{n}\}$ and $\{e_{2,j,m} \in \af{\theta,n} : 1\leq j,m \leq q^\theta_{n-1} \}$  be the standard family of matrix units in, respectively, $\alg{M}(q^\theta_{n})$ and $\alg{M}(q^\theta_{n-1})$ inside $\af{\theta,n} = \alg{M}(q^\theta_{n}) \oplus \alg{M}(q^\theta_{n-1})$, as in Section (\ref{exp-formula-sec}) and with $(p_n^\theta)_{n\in\N}$ and $(q_n^\theta)_{n\in\N}$ defined by Expression (\ref{pq-rel-eq}).

For $1 \leq j,m \leq q^\theta_{n} $, we compute:
\begin{equation*}
\sigma_\theta \left( \indmor{\alpha}{n}(e_{1,j,m}^\ast  e_{1,j,m}) \right)= (-1)^{n-1}(\theta q^\theta_{n-1} -p^\theta_{n-1})
\end{equation*}
while, for $ 1 \leq j,m \leq q^\theta_{n-1}$:
\begin{equation*}
\sigma_\theta ( \indmor{\alpha}{n}(e_{2,j,m}^\ast e_{2,j,m}) )= (-1)^{n}(\theta q^\theta_{n} -p^\theta_{n} ).
\end{equation*}
\end{lemma}

\begin{proof}  
Let $ 1\leq j,m \leq q^\theta_{n} $. By Lemma (\ref{af-theta-trace-lemma}), we have:
\begin{equation*}
\begin{split}
\sigma_\theta \left( \indmor{\alpha}{n}(e_{1,j,m}^\ast  e_{1,j,m}) \right) & = t(\theta,n) \mathsf{tr}_{q^\theta_{n}} (e_{1,j,m}^\ast  e_{1,j,m})  + (1-t(\theta,n))\cdot 0 \\
&=\frac{t(\theta,n)}{q^\theta_{n}} \\
&=(-1)^{n-1}(\theta q^\theta_{n-1} -p^\theta_{n-1} ).
\end{split}
\end{equation*}
And, a similar argument proves the result for the other matrix units.
\end{proof}

\subsection{Continuity of $\theta \in (0,1)\setminus\Q \mapsto \af{\theta}$}

Our proof that the map $\theta\in(0,1)\setminus\Q\mapsto (\af{\theta},\Lip_\theta)$ is continuous for the quantum propinquity relies on a homeomorphism between the Baire space of Definition (\ref{Baire-Space-def}) and $(0,1)\setminus\Q$, endowed with its topology as a subspace of $\R$. Indeed, the map which associates, to an irrational number in $(0,1)$, its continue fraction expansion is a homeomorphism (see, for instance, \cite{Miller95}). We include a brief proof of this fact as, while it is well-known, the proof is often skipped in references. Moreover, this will serve as a means to set some other useful notations for our work.

\begin{notation}
Define $\mathsf{cf} : (0,1)\setminus\Q \rightarrow \BaireSpace$ by setting $\mathsf{cf}(\theta) = (b_n)_{n\in\N}$ if and only if $\theta = [0, b_0, b_1, \ldots]$. We note that $\mathsf{cf}$ is a bijection from $(0,1)\setminus\Q $ onto $\BaireSpace$, where $\BaireSpace$ is the Baire space defined in Definition (\ref{Baire-Space-def}). The inverse of $\mathsf{cf}$ is denote by $\mathrm{ir} : \BaireSpace \rightarrow (0,1)\setminus\Q$. 
\end{notation} 

\begin{notation}
We will denote the closed ball in $(\BaireSpace,\mathsf{d})$ of center $x \in \BaireSpace$ and radius $2^{-N}$ by $\BaireSpace[x,N]$ for $N>0$. It consists of all sequences in $\BaireSpace$ whose $N$ first entries are the same as the $N$ first entries of $x$.  
\end{notation}

\begin{proposition}  
The bijection:
\begin{equation*}
\mathsf{cf} :  ((0,1) \setminus \mathds{Q}, \vert \cdot \vert  ) \longrightarrow  (\mathscr{N}, \mathsf{d})
\end{equation*}
is a homeomorphism.
\end{proposition}

\begin{proof}  
The basic number theory facts used in this proof can be found in \cite{Hardy38}.  Since every irrational in $(0,1)$ has a unique continued fraction expansion of the form given by Expression (\ref{continued-fraction-eq}), and every sequence of positive integers determines the continued fraction expansion of an irrational via the same expression, $\mathsf{cf}$   is a bijection.  

We now show that $\mathsf{cf}  $ is continuous. 

Let $b = (b_n)_{n\in\N} \in \BaireSpace$ and let:
\begin{equation*}
\theta = \lim_{n\rightarrow\infty} \cfrac{1}{b_0 + \cfrac{1}{b_1 + \cfrac{1}{\ddots + \cfrac{1}{b_n}}}} \in (0,1)\setminus\Q\text{.}
\end{equation*}
Let $V = \BaireSpace[b,N]$ for some $N \in\N\setminus\{0\}$. 

Let $\eta \in \mathsf{cf}^{-1}(V)$ and let $(x_n)_{n\in\N} = \mathsf{cf}(\eta)$. Thus, for all $j\in \{0,\ldots,N-1\}$, we have $x_n = b_n$. Define $I_{N,\eta}$ as the open interval with end points:
\begin{equation*}
\cfrac{1}{b_0 + \cfrac{1}{ b_1 + \cfrac{1}{\ddots + \frac{1}{b_{N-1}}}}}\;\text{ and }\;\cfrac{1}{b_0 + \cfrac{1}{ b_1 + \cfrac{1}{\ddots + \frac{1}{b_{N-1}\,+1}}}}\text{,}
\end{equation*}
and let $\Theta_{N,\eta} = I_{N,\eta}\setminus\Q$. 

By construction, $\Theta_{N,\eta}$ is open in the relative topology on $(0,1) \setminus \mathds{Q}$, and since $\eta$ is irrational, we conclude $\eta \in \Theta_{N, \eta} \setminus \mathds{Q}$.  Furthermore, $\mathsf{cf}( \Theta_{N,\eta})\subseteq V$, which concludes the argument since the set of open balls in $\BaireSpace$ is a topological basis for $\BaireSpace$.

Next, we show continuity of $\mathsf{ir}$ by sequential continuity. Let $(b^n)_{n\in\N}$ be a sequence in $\BaireSpace$, where, for all $n\in\N$, we write $b^n = (b^n_m)_{m\in\N}$. Assume $(b^n)_{n\in\N}$ converges to some $b\in \BaireSpace$ for $\mathsf{d}$.

For each $n \in \N$, define: 
\begin{equation*}
 N(b^n , b)=\min \{ m \in \N \cup \{ \infty\} :   b^n_m \neq b_m \} ,
\end{equation*} 
in which $N(b^n , b)=\infty $ if $b^n =b$. Let $\theta = \mathsf{ir}(b) \in (0,1)\setminus\Q$. 
By Definition (\ref{Baire-Space-def}) of our metric $\mathsf{d}$ on $\BaireSpace$, we conclude that:
\begin{equation}\label{Baire-homeo-eq1}
\lim_{ n \to \infty} N(b^n , b) = \infty \text{.}
\end{equation}
We choose, in particular, $M \in \N$ such that for all $n\geq M$, we have $N(b^n,b) \geq 1$.

We also note that if $\theta_n = \mathsf{ir}(b^n)$ for $n\in\N$, then, using Notation (\ref{pq-rel-eq}), we conclude that $p_m^\theta = p_m^{\theta_n}$ and $q_m^\theta = q_m^{\theta_m}$ for all $m \in \{0,\ldots,N(b^n,b)-1\}$. Thus, for all $n\in \N$ with $n\geq M$, standard estimates for continued fraction expansions lead to:

\begin{equation*}
\begin{split}
\vert \mathsf{ir}(b^n)-\theta \vert &= \left\vert \mathsf{ir}(b^n)-p^\theta_{N(b^n , b) -1}/q^\theta_{N(b^n , b) -1} +p^\theta_{N(b^n , b) -1}/q^\theta_{N(b^n , b) -1}-\theta \right\vert\\
& \leq  \left\vert \mathsf{ir}(b^n)-p^\theta_{N(b^n , b) -1}/q^\theta_{N(b^n , b) -1} \right\vert + \left\vert p^\theta_{N(b^n , b) -1}/q^\theta_{N(b^n , b) -1}-\theta\right\vert  \\
& =  \left\vert \theta_n -p^{\theta_n}_{N(b^n , b) -1}/q^{\theta_n}_{N(b^n , b) -1} \right\vert +  \left\vert p^\theta_{N(b^n , b) -1}/q^\theta_{N(b^n , b) -1}-\theta\right\vert \\
& < 1/\left(q^{\theta_n}_{N(b^n , b) -1} \right)^2 + 1/\left(q^\theta_{N(b^n , b) -1}\right)^2 \\
& = 2/\left(q^\theta_{N(b^n , b) -1}\right)^2 \text{.}
\end{split}
\end{equation*}
 Thus by Equation (\ref{Baire-homeo-eq1}), we conclude that $\lim_{n\rightarrow\infty} \mathsf{ir}(b^n) = \theta = \mathsf{ir}(b)$ as desired, and our proof is complete.
\end{proof}

Our main result will be proven in four steps. We begin by observing that the tracial states of $\af{\theta}$ provide a continuous field of states on various finite dimensional algebras.

\begin{lemma}\label{trace-continuity-lemma}
Let $\theta \in (0,1)\setminus\Q$ and $N\in\N$. Let $(p_n^\theta)_{n\in\N}$ and $(q_n^\theta)_{n\in\N}$ be defined from $\mathsf{cf}(\theta)$ using Expression (\ref{pq-rel-eq}). For all $n\in\{0,\ldots,N\}$, the map:
\begin{equation}\label{trace-continuity-lemma-eq1}
\mathrm{s}_n : (x,a) \in \BaireSpace[\mathsf{cf}(\theta),N+1] \times \af{\theta,n} \longmapsto \sigma_{\mathsf{ir}(x)} \left(\indmor{\alpha_{\mathrm{ir}(x)}}{n}(a)\right)
\end{equation}
is well-defined and continuous from $\BaireSpace[\mathsf{cf}(\theta),N+1] \times (\af{\theta,n},\|\cdot\|_{\af{\theta,n}})$ to $\R$.
\end{lemma}

\begin{proof}
Let $x, y \in \BaireSpace[\mathsf{cf}(\theta),N]$ and set $\eta = \mathsf{ir}(x)$ and $\xi = \mathsf{ir}(y)$. Since $\mathsf{d}$ is an ultrametric on $\BaireSpace$, we note that $\mathsf{d}(x,y) \leq \frac{1}{2^N}$.

We note that the result is trivial for $n=0$ since $\mathsf{s}_0$ is the identity on $\C=\af{x,0}$ for all $x\in\BaireSpace$.

We now use the notation of Expression (\ref{pq-rel-eq}). The key observation from Expression (\ref{pq-rel-eq}) is that the functions:
\begin{equation*}
z \in \BaireSpace[\mathsf{cf}(\theta), N+1] \mapsto \left(q_n^{\mathrm{ir}(z)}, p_n^{\mathrm{ir}(z)}\right)
\end{equation*}
are constant for all $n\in \{0,\ldots,N\}$, equal to $(q_n^\theta, p_n^\theta)$ --- since $\mathsf{d}(x,\mathsf{cf}(\theta))\leq \frac{1}{2^{N+1}}$ implies that the sequences $x$ and $\mathsf{cf}(\theta)$ agree on their first $N$ entries. 

Thus, setting $\B_n = \af{\theta,n}$, we have:
\begin{equation*}
\alg{M}(q_{n}^{x})\oplus\alg{M}(q_{n-1}^{x}) = \B_n
\end{equation*}
for all $n\in\{0,\ldots,N\}$, and the maps defined by Expression (\ref{trace-continuity-lemma-eq1}) are well-defined.

Let now $n\in \{1,\ldots,N\}$ be fixed. Let $a\in \alg{B}_n$ and write $a = a'\oplus a'' \in \alg{M}(q^\theta_{n}) \oplus\alg{M}(q^\theta_{n-1})$. By Lemma (\ref{af-theta-trace-lemma}), we compute:
\begin{equation*}
\begin{split}
\left|\sigma_\xi\circ\indmor{\alpha_\xi}{n}(a) - \sigma_\eta\circ\indmor{\alpha_\eta}{n}(a)\right| &= |(t(\xi,n) - t(\eta,n))(\mathrm{tr}_{q^\theta_{n}}(a') - \mathrm{tr}_{q^\theta_{n-1}}(a''))|\\
&\leq 2|t(\xi,n) - t(\eta,n)| \|a\|_{\B_n}\\
&= 2|q^\theta_{n}(\xi q^\theta_{n-1} - p^\theta_{n-1}) - q^\theta_{n}(\eta q^\theta_{n-1} - p^\theta_{n-1})| \|a\|_{\B_n}\\
&= 2|q^\theta_{n}q^\theta_{n-1}| |\xi-\eta|\|a\|_{\B_n}\\
&= 2|q^\theta_{n}q^\theta_{n-1}| |\mathrm{ir}(y) - \mathrm{ir}(x)|\|a\|_{\B_n}\text{.}
\end{split}
\end{equation*}
As $n < N$ is fixed, and $\mathrm{ir}$ is a homeomorphism, we conclude that if $(y_m)_{m\in\N}$ is a sequence in $\BaireSpace[\theta,N+1]$ converging to $x$ then:
\begin{equation*}
\lim_{m\rightarrow\infty}  \left|\sigma_{\mathsf{ir}(y_m)}\circ\indmor{\alpha_{\mathrm{ir}(y_m)}}{n}(a) - \sigma_\eta\circ\indmor{\alpha_\eta}{n}(a)\right| = 0 \text{.}
\end{equation*}
Thus we have established that the partial function $s_n(\cdot,a)$ are continuous for all $a\in\B_n$.

We now prove the joint continuity of our maps. Let $a,b \in \B_n$ and $\eta,\xi$ as above. Then:
\begin{multline*}
\left\vert \sigma_{\eta} \left(\indmor{\alpha_\eta}{n}(a)\right)  - \sigma_{\xi} \left(\indmor{\alpha_\xi}{n}(b)\right) \right\vert \\
\begin{split}
&= \left\vert \sigma_{\eta} \left(\indmor{\alpha_\eta}{n}(a)\right) -\sigma_\eta \left(\indmor{\alpha_\eta}{n}(b)\right)+\sigma_{\eta} \left(\indmor{\alpha_\eta}{n}(b)\right) - \sigma_\xi \left(\indmor{\alpha_\xi}{n}(b)\right) \right\vert \\
& \leq  \left\vert \sigma_{\eta} \left(\indmor{\alpha_\eta}{n}(a)\right) -\sigma_\eta \left(\indmor{\alpha_\eta}{n}(b)\right)\right\vert \\
&\quad + \left\vert \sigma_{\eta} \left(\indmor{\alpha_\eta}{n}(b)\right) - \sigma_\xi \left(\indmor{\alpha_\xi}{n}(b)\right) \right\vert \\
&  \leq \Vert a -b \Vert_{\A_n} + \left\vert \sigma_{\eta} \left(\indmor{\alpha_\eta}{n}(b)\right) - \sigma_\xi \left(\indmor{\alpha_\xi}{n}(b)\right) \right\vert\text{.}
\end{split}
\end{multline*}
It follows immediately that the map $\mathsf{s}_n$ defined by Expression (\ref{trace-continuity-lemma-eq1}) is continuous as desired.
\end{proof}

Our second step is to prove that, thanks to Lemma (\ref{trace-continuity-lemma}), the Lip-norms induced from $\af{\theta}$ on their finite dimensional C*-subalgebras form a continuous field of Lip-norms \cite{Rieffel00}. Moreover, we obtain a joint continuity result for these Lip-norms, which are thus in particular continuous rather than only lower semi-continuous.

\begin{lemma}\label{lip-norm-continuity-lemma}
Let $\theta \in (0,1)\setminus\Q$ and $N\in\N$. Let $(p_n^\theta)_{n\in\N}$ and $(q_n^\theta)_{n\in\N}$ be defined from $\mathsf{cf}(\theta)$ using Expression (\ref{pq-rel-eq}). For all $n\in\{0,\ldots,N\}$ and $k\in (0,\infty)$, the map:
\begin{equation}\label{lip-norm-continuity-lemma-eq1}
\mathsf{l}_n : (x,a) \in \BaireSpace[\mathsf{cf}(\theta),N+1] \times \af{\theta,n} \longmapsto\Lip^k_{\mathsf{ir}(x)} \left(\indmor{\alpha_{\mathrm{ir}(x)}}{n}(a)\right)
\end{equation}
defined using Notation (\ref{af-theta-lip-norm-notation}), is well-defined and continuous from $\BaireSpace[\mathsf{cf}(\theta),N+1]  \times (\B_n,\|\cdot\|_{\B_n})$ to $\R$.
\end{lemma}

\begin{proof} 
We note that the proof of Lemma (\ref{trace-continuity-lemma}) also establishes, by a similar argument, that the maps $\mathsf{l}_n$ are well-defined for all $n\in \{0,\ldots,N\}$. We also note that $\mathsf{l}_0$ is constantly $0$, and thus the result is trivial for $n=0$.

Fix $n\in \{1,\ldots,N\}$. Let $x,y \in \BaireSpace[\mathsf{cf}(\theta),N+1]$ and write $\eta = \mathrm{ir}(x)$ and $\xi = \mathrm{ir}(y)$. As within the proof of Lemma (\ref{trace-continuity-lemma}), we note that for all $M \in \{0, \ldots, n\}$, we have that  and set $q_M=q_M^\theta = q_M^\eta = q_M^\xi$ and similarly, $p_M=p_M^\theta = p_M^\eta = p_M^\xi$ (using the notations of Expression (\ref{pq-rel-eq}) ). Furthermore, for all $M \in \{0,\ldots , n\}$, we set $(\B_M , \alpha_M)= (\af{\theta,M}, \alpha_{\theta,M}) = (\af{\eta,M}, \alpha_{\eta, M} )= (\af{\xi,M}, \alpha_{\xi,M})$ .  Note further that $\alpha_{M,n-1}=\alpha_{n-1} \circ \cdots \circ \alpha_M = \alpha_{\theta,M,n-1}= \alpha_{\eta,M,n-1}=\alpha_{\xi,M,n-1}$ for all $M \in \{0,\ldots, n-1\}$.

Fix $M \in \{0,\ldots, n-1\}$, we employ the notations of Section (\ref{exp-formula-sec}) and thus, we have a set $\{e_{1,j,m} \in \B_M : 1\leq j,m \leq q_{M} \}$ of matrix units of $\alg{M}(q_{M})\subseteq \B_M$ and a set $\{e_{2,j,m} \in \B_M : 1\leq j,m \leq q_{M-1}\}$ of matrix units for $\alg{M}(q_{M-1}) \subseteq \B_M$. 

To lighten our notations in this proof, let:
\begin{equation*}
\begin{split}
I_1 & =\{ (1,j,m)\in \N^3 : 1\leq j,m \leq q_{M} \}\text{, } I_2 =  \{ (2,j,m)\in \N^3 : 1\leq j,m \leq q_{M-1} \}\\
& \text{and} \  I=I_1\cup I_2\text{.}
\end{split}
\end{equation*}
Let $ a\in \B_n$. By Expression (\ref{cond-exp-eq}) and the same argument provided by Equation (\ref{induct-trace}) in the proof of Theorem (\ref{uhf-thm}), we conclude that: 
\begin{equation*}
\left\| \indmor{\alpha_\eta}{n}(a) - \CondExp{\indmor{\alpha_\eta}{n}(a)}{\indmor{\alpha_\eta}{M}(\B_M)}\right\|_{\af{\eta}} =\left\| a - \sum_{j \in I} \frac{\sigma_\eta\left(\indmor{\alpha_\eta}{n}(\alpha_{M,n-1}( e_j^\ast)a)\right)}{\sigma_\eta\left(\indmor{\alpha_\eta}{M}(e_j^\ast e_j)\right)} \alpha_{M,n-1}(e_j) \right\|_{\B_n}
\end{equation*}
and
\begin{equation*}
\left\| \indmor{\alpha_\xi}{n}(a) - \CondExp{\indmor{\alpha_\xi}{n}(a)}{\indmor{\alpha_\xi}{M}(\B_M)}\right\|_{\af{\xi}} =\left\| a - \sum_{j \in I} \frac{\sigma_\xi\left(\indmor{\alpha_\xi}{n}(\alpha_{M,n-1}( e_j^\ast)a)\right)}{\sigma_\xi\left(\indmor{\alpha_\xi}{M}(e_j^\ast e_j)\right)} \alpha_{M,n-1}(e_j) \right\|_{\B_n}.
\end{equation*}
Next, let  $a, b \in \B_n$, we have:
\begin{align*}
&\quad \left| \left\| \indmor{\alpha_\eta}{n}(a) - \CondExp{\indmor{\alpha_\eta}{n}(a)}{\indmor{\alpha_\eta}{M}(\B_M)}\right\|_{\af{\eta}}  -\left\| \indmor{\alpha_\xi}{n}(b) - \CondExp{\indmor{\alpha_\xi}{n}(b)}{\indmor{\alpha_\xi}{M}(\B_M)}\right\|_{\af{\xi}} \right| \\
&\leq \Vast| \left\| a - \sum_{j \in I} \frac{\sigma_\eta\left(\indmor{\alpha_\eta}{n}(\alpha_{M,n-1}( e_j^\ast)a)\right)}{\sigma_\eta\left(\indmor{\alpha_\eta}{M}(e_j^\ast e_j)\right)} \alpha_{M,n-1}(e_j) \right\|_{\B_n} \\
& \quad - \left\| b - \sum_{j \in I} \frac{\sigma_\xi\left(\indmor{\alpha_\xi}{n}(\alpha_{M,n-1}( e_j^\ast)b)\right)}{\sigma_\xi\left(\indmor{\alpha_\xi}{M}(e_j^\ast e_j)\right)} \alpha_{M,n-1}(e_j) \right\|_{\B_n} \Vast| \\
& \leq \|a-b\|_{\B_n} + \left\|\sum_{j \in I_1} \left(\frac{\sigma_\eta\left(\indmor{\alpha_\eta}{n}(\alpha_{M,n-1}( e_j^\ast)a)\right)}{q_{M-1}^\eta \eta - p_{M-1}^\eta} - \frac{\sigma_\xi\left(\indmor{\alpha_\xi}{n}(\alpha_{M,n-1}( e_j^\ast)b)\right)}{q_{M-1}^\xi \xi - p_{M-1}^\xi}\right)e_j \right\|_{\M(q_{M})}\\
&\quad + \left\|\sum_{j \in I_2} \left(\frac{\sigma_\eta\left(\indmor{\alpha_\eta}{n}(\alpha_{M,n-1}( e_j^\ast)a)\right)}{q_{M}^\eta \eta - p_{M}^\eta} - \frac{\sigma_\xi\left(\indmor{\alpha_\xi}{n}(\alpha_{M,n-1}( e_j^\ast)b)\right)}{q_{M}^\xi \xi - p_{M}^\xi}\right)e_j \right\|_{\M(q_{M-1})}\\
& = \|a-b\|_{\B_n} + \left\|\sum_{j \in I_1} \left(\frac{\mathrm{s}_n\left(x,\alpha_{M,n-1}( e_j^\ast)a\right)}{q_{M-1}\mathrm{ir}(x) - p_{M-1}} - \frac{\mathrm{s}_n\left(y,\alpha_{M,n-1}( e_j^\ast)b\right)}{q_{M-1}\mathrm{ir}(y) - p_{M-1}}\right)e_j \right\|_{\M(q_M)}\\
&\quad + \left\|\sum_{j \in I_2} \left(\frac{\mathrm{s}_n\left(x,\alpha_{M,n-1}( e_j^\ast)a\right)}{q_{M}\mathrm{ir}(x) - p_{M}} - \frac{\mathrm{s}_n\left(y,\alpha_{M,n-1}( e_j^\ast)b\right)}{q_{M}\mathrm{ir}(y) - p_{M}}\right)e_j \right\|_{\M(q_{M-1})},\\
\end{align*}
where we used Lemma (\ref{mu-norm-lemma}) in the second inequality above, and $\mathrm{s}_n $ is defined by Expression (\ref{trace-continuity-lemma-eq1}).  Now, since $\mathrm{ir}$ is a homeomorphism from $\BaireSpace$ to the irrationals in $(0,1)$, and the map $\mathrm{s}_n $ is continuous by Lemma (\ref{trace-continuity-lemma}), we conclude that as $I=I_1 \cup I_2$ is finite: 
\begin{equation}\label{lip-norm-continuity-lemma-eq2}
(x,a) \in \BaireSpace[\mathsf{cf}(\theta),N+1] \times \B_n \longmapsto \frac{1}{\beta(M)} \left\| \indmor{\alpha_{\mathrm{ir}(x)}}{n}(a) - \CondExp{\indmor{\alpha_{\mathrm{ir}(x)}}{n}(a)}{\indmor{\alpha_{\mathrm{ir}(x)}}{M}(\B_M)} \right\|_{\af{\mathrm{ir}(x)}}
\end{equation}
is continuous, where $\beta(M) = \frac{1}{\left((q_M)^2 + (q_{M-1})^2 \right)^k}$.

Last, we note that since for all $j \geq n$ we have:
\begin{equation*}
\CondExp{\indmor{\alpha_{\mathrm{ir}(x)}}{n}(a)}{\indmor{\alpha_{\mathrm{ir}(x)}}{j}(\af{\theta,j})} = \indmor{\alpha_{\mathrm{ir}(x)}}{n}(a)
\end{equation*}
by definition of conditional expectation, and therefore, the function $\mathsf{l}_n$ is the maximum of the functions given in Expression (\ref{lip-norm-continuity-lemma-eq2}) with $M$ ranging over $\{0,\ldots,n-1\}$.

As the maximum of finitely many continuous functions is continuous, our lemma is proven.
\end{proof}

Our third step establishes a bound for the propinquity between finite dimensional quantum compact metric spaces which constitute the building blocks of the C*-algebras $\af{\theta}$.

\begin{lemma}\label{af-theta-fd-lemma} 
Let $\theta \in (0,1)\setminus\Q$ and $N\in\N$. Let $(p_n^\theta)_{n\in\N}$ and $(q_n^\theta)_{n\in\N}$ be defined from $\mathsf{cf}(\theta)$ using Expression (\ref{pq-rel-eq}). For all $n\in\{0,\ldots,N\}$ and $k\in (0,\infty)$, setting $\B_n = \af{\theta,n}$, the map:
\begin{equation}\label{af-theta-fd-lemma-eq1}
\mathsf{q}_n : x \in \BaireSpace[\mathsf{cf}(\theta),N+1] \longmapsto \left(\B_n, \Lip^k_{\mathsf{ir}(x)} \circ\indmor{\alpha_{\mathrm{ir}(x)}}{n}\right)
\end{equation}
defined using Notation (\ref{af-theta-lip-norm-notation}), is well-defined and continuous from $(\BaireSpace,\mathsf{d})$ to the class of $(2,0)$-quasi-Leibniz quantum compact metric spaces metrized by the quantum propinquity $\qpropinquity{}$.
\end{lemma}

\begin{proof}
The statement is obvious for $n = 0$.

Let $n \in \{1,\ldots,N\}$. Let $\alg{W}$ be any complementary subspace of $\R\unit_\A$ in $\sa{\B_n}$ --- which exists since $\sa{\B_n}$ is finite dimensional. We shall denote by $\alg{S}$ the unit sphere $\{ a\in \alg{W} : \|a\|_{\B_n} = 1 \}$ in $\alg{W}$. Note that since $\alg{W}$ is finite dimensional, $\alg{S}$ is a compact set.

We let $x\in \BaireSpace[\mathsf{cf}(\theta),N+1]$. Let $(y_m)_{m\in\N}$ be a sequence in $\BaireSpace[\mathsf{cf}(\theta),N+1]$ converging to $x$. Let:
\begin{equation*}
S = \left\{ x, y_m : m\in\N \right\} \times \alg{S}
\end{equation*}
which is a compact subset of $\BaireSpace\times \alg{W}$. Since the function:
\begin{equation*}
\mathsf{l}_n : (u,a) \in \BaireSpace[\mathsf{cf}(\theta),N] \times \B_n \longmapsto\Lip^k_{\mathsf{ir}(u)} \left(\indmor{\alpha_{\mathrm{ir}(u)}}{n}(a)\right)
\end{equation*}
is continuous by Lemma (\ref{lip-norm-continuity-lemma}),  $\mathsf{l}_n$ reaches a minimum on $S$: thus there exists $(z,c) \in S$ such that $\min_{S}\mathsf{l}_n = \mathsf{l}_n(z,c)$. In particular, since Lip-norms are zero only on the scalars, we have $\mathsf{l}_n(z,c) > 0$ as $\|c\|_{\alg{W}} = 1$ yet the only scalar multiple of $\unit_{\B_n}$ in $\alg{W}$ is $0$. We denote $m_S = \mathsf{l}_n(z,c) > 0$ in the rest of this proof.

Moreover, $\mathsf{l}_n$ is continuous on the compact $S$ so it is uniformly continuous.

Let $\varepsilon > 0$. As $\mathsf{l}_n$ is uniformly continuous on $S$, there exists $M \in \N$ such that for all $m\geq M$ and for all $a \in \alg{S}$ we have:
\begin{equation*}
|\mathsf{l}_n(y_m, a) - \mathsf{l}_n(x,a)| \leq m_S^2 \varepsilon\text{.}
\end{equation*}

We then have, for all $a\in\alg{S}$ and $m\geq M$:

\begin{equation*}
\begin{split}\left\| a - \frac{\mathsf{l}_n(y_m,a)}{\mathsf{l}_n(x,a)} a \right\|_{\B_n} &= \frac{|\mathsf{l}_n(y_m,a)-\mathsf{l}_n(x,a)|}{\mathsf{l}_n(x,a)} \|a\|_{\B_n} \\
&\leq \frac{\varepsilon m_S^2}{m_S} \leq m_S \varepsilon \text{.}
\end{split}
\end{equation*}
Similarly:
\begin{equation}\label{af-theta-fd-lemma-eq2}
\left\| a - \frac{\mathsf{l}_n(x,a)}{\mathsf{l}_n(y_m,a)} a \right\|_{\B_n} \leq m_S \varepsilon \text{.}
\end{equation}

We are now ready to provide an estimate for the quantum propinquity. Let $m\geq M$ be fixed. Writing $\mathrm{id}$ for the identity of $\B_n$, the quadruple:
\begin{equation*}
\gamma = \left(\B_n, \unit_{\B_n}, \mathrm{id}, \mathrm{id}\right)
\end{equation*}
is a bridge from $\left(\B_n,\Lip_{\mathrm{ir}(y_m)}^k\circ\indmor{\alpha_{\mathrm{ir}(y_m)}}{n}\right)$ to $\left(\B_n,\Lip_{\mathrm{ir}(x)}^k\circ\indmor{\alpha_{\mathrm{ir}(x)}}{n}\right)$.

As the pivot of $\gamma$ is the unit, the height of $\gamma$ is null. We are left to compute the reach of $\gamma$.

Let $a\in \B_n$. We proceed in four steps.

\setcounter{step}{0}
\begin{step}
Assume that $a \in \R\unit_{\B_n}$. 
\end{step}
We then have that $\mathsf{l}_n(y_m,a) = 0$ as well, and that $\|a - a\|_{\B_n} = 0$.

\begin{step}
Assume that $a \in\alg{S}$.
\end{step}
We note again that $\mathsf{l}_n(x,a) \geq m_S > 0$. By Inequality (\ref{af-theta-fd-lemma-eq2}), we note that:
\begin{equation*}
\left\|a - \frac{\mathsf{l}_n(x,a)}{\mathsf{l}_n(y_m,a)} a\right\|_{\B_n}\leq \varepsilon m_S \leq \varepsilon \mathsf{l}_n(x,a) \text{,}
\end{equation*}
while $\mathsf{l}_n\left(y_m, \frac{\mathsf{l}_n(x,a)}{\mathsf{l}_n(y_m,a)} a \right) = \mathsf{l}_n(x,a)$.
 
\begin{step}
Assume that $a = b + t\unit_{\B_n}$ with $b\in\alg{S}$.
\end{step}

Note that $\mathsf{l}_n(x,b) = \mathsf{l}_n(x,a)$. Therefore, let $b' \in \sa{\B_n}$ be constructed as in Step 2. We then check easily that:
\begin{equation*}
\|a - (b' + t\unit_{\B_n})\|_{\B_n} = \|b - b'\|_{\alg{B}_n} \leq \varepsilon \mathsf{l}_n(x,a)
\end{equation*}
while $\mathsf{l}_n(y_m, b'+t\unit_{\B_n}) = \mathsf{l}_n(y_m,b') \leq \mathsf{l}_n(x,a)$.
\begin{step}
Let $a\in \sa{\B_n}$.
\end{step}

By definition of $\alg{S}$ there exists $r, t \in \R$ such that $a = rb + t\unit_{\B_N}$ with $b\in \alg{S}$. Let $b' \in\sa{\A}$ be constructed from $b$ as in Step 3. Then set $a' = rb'$. By Step 3, we have $\mathsf{l}_n(y_m,b') \leq\mathrm{l}_n(x,a')$ and $\|a'-b'\|_{\B_n} \leq\varepsilon \mathsf{l}_n(x,a')$.

Thus by homogeneity, we conclude that:
\begin{equation}\label{af-theta-fd-lemma-eq3}
\forall a\in \sa{\B_n}\quad\exists a' \in\sa{\B_n}\quad \|a-a'\|_{\B_n}\leq\varepsilon \mathsf{l}_n(x,a)\text{ and }\mathsf{l}_n(y_m,a')\leq\mathsf{l}_n(x,a)\text{.}
\end{equation}

By symmetry in the roles of $x$ and $y_m$ we can conclude as well that:
\begin{equation}\label{af-theta-fd-lemma-eq4}
\forall a\in \sa{\B_n}\quad\exists a' \in\sa{\B_n}\quad \|a-a'\|_{\B_n}\leq\varepsilon \mathsf{l}_n(y_m,a)\text{ and }\mathsf{l}_n(x,a')\leq\mathsf{l}_n(y_m,a)\text{.} 
\end{equation}

Now, Expressions (\ref{af-theta-fd-lemma-eq3}) and (\ref{af-theta-fd-lemma-eq4}) together imply that the reach, and hence the length of the bridge $\gamma$ is no more than $\varepsilon$.

Therefore, for all $m\geq M$, we have:
\begin{equation*}
\qpropinquity{}\left((\B_n,\mathsf{l}_n(x,\cdot)), (\B_n, \mathsf{l}_n(y_m,\cdot))\right) \leq \varepsilon
\end{equation*}
which concludes our proof.
\end{proof}

We are now able to establish the main result of this section.

\begin{theorem}\label{af-theta-thm}
For all $k\in (0,\infty)$ and using Notations (\ref{af-theta-notation}) and (\ref{af-theta-lip-norm-notation}), the function:
\begin{equation*}
\theta \in (0,1)\setminus \Q \longmapsto \left(\af{\theta}, \Lip_\theta^k\right) \in \mathcal{AF}^k 
\end{equation*}
is continuous from $(0,1)\setminus\Q$, with its topology as a subset of $\R$, to the class of $(2,0)$-quasi-Leibniz quantum compact metric spaces metrized by the quantum propinquity $\qpropinquity{}$.
\end{theorem}

\begin{proof}
\setcounter{step}{0}

The golden ratio $\phi = \frac{1+\sqrt{5}}{2}$ and $\Phi = \phi - 1 = \frac{1}{\phi}$ be its reciprocal. The continued fraction expansion of $\Phi$ is given by:
\begin{equation*}
\Phi = \cfrac{1}{1 + \cfrac{1}{1 + \cfrac{1}{1 + \cfrac{1}{1 + \ddots}}}}
\end{equation*}
and $\af{\Phi}$ is sometimes called the Fibonacci C*-algebra \cite{Davidson}. Its importance for our work is that the associated sequence $(q_n^\Phi)_{n\in\N}$ defined by Expression (\ref{pq-rel-eq}) is the least possible sequence of the form $(q_n^\theta)_{n\in\N}$ given by the same expression, over all possible $\theta \in (0,1)\setminus\Q$ (where the order is defined entry-wise).

Let $\theta \in (0,1)\setminus\Q$. By Theorem (\ref{AF-lip-norms-thm}), we have for all $n\in\N$:
\begin{equation}\label{af-theta-thm-eq1}
\qpropinquity{}\left((\af{\theta},\Lip_\theta^k), (\af{\theta,n},\mathsf{l}_n(\theta,\cdot))\right) \leq \left(\frac{1}{(q_n^\theta)^2 + (q_{n-1}^\theta)^2}\right)^k \leq \left(\frac{1}{(q_n^\Phi)^2 + (q_{n-1}^\Phi)^2}\right)^k \text{,}
\end{equation}
where $\mathsf{l}_n$ is defined in Lemma (\ref{lip-norm-continuity-lemma}).

We are now in a position to conclude. Let $(\theta_m)_{m\in\N}$ be a sequence in $(0,1)\setminus\Q$ converging to $\theta$. Let $\varepsilon > 0$.

To begin with, let $N\in\N$ such that for all $n\geq N$, we have:
\begin{equation*}
\left(\frac{1}{(q_n^\Phi)^2 + (q_{n-1}^\Phi)^2}\right)^k \leq \frac{\varepsilon}{2}\text{.}
\end{equation*}

We thus have, for all $m\in\N$, that:
\begin{equation}\label{af-theta-thm-eq2}
\qpropinquity{}\left((\af{\theta},\Lip_\theta^k),(\af{\theta_m},\Lip_{\theta_m}^k)\right) \leq \varepsilon + \qpropinquity{}\left((\af{\theta,N},\mathsf{l}_N(\theta,\cdot)),(\af{\theta_m,N},\mathsf{l}_N(\theta_m,\cdot))\right)\text{.}
\end{equation}

Now, let $x_m = \mathsf{cf}(\theta_m)$ for all $m\in\N$ and $x=\mathsf{cf}(\theta)$. Since $\mathsf{cf}$ is a continuous, the sequence $(x_m)_{m\in\N}$ converges to $x$ in $\BaireSpace$. Thus there exists $M_1 \in\N$ such that, for all $m\geq M_1$, we have $\mathsf{d}(x,x_m) \leq \frac{1}{2^{N+1}}$, i.e. $x_m \in \BaireSpace[x,N+1]$.

We thus apply Lemma (\ref{af-theta-fd-lemma}) to obtain from Expression (\ref{af-theta-thm-eq2}) that:
\begin{equation*}
\qpropinquity{}\left((\af{\theta},\Lip_\theta^k),(\af{\theta_m},\Lip_{\theta_m}^k)\right) \leq \varepsilon + \qpropinquity{} \left(\mathsf{q}_N(\theta), \mathsf{q}_N(\theta_m)\right) \text{.}
\end{equation*}
Now, Lemma (\ref{af-theta-fd-lemma}) establishes that $\mathsf{q}_N$ is continuous. Hence:
\begin{equation*}
\limsup_{m\rightarrow\infty} \qpropinquity{}\left((\af{\theta},\Lip_\theta^k),(\af{\theta_m},\Lip_{\theta_m}^k)\right)\leq \varepsilon\text{.}
\end{equation*}
As $\varepsilon>0$ was arbitrary, our Theorem is proven.
\end{proof}

\section{Some compactness results for AF algebras}

The search for compact classes of quantum compact metric spaces for the quantum propinquity is a delicate yet interesting challenge. The main result on this topic is given by the following analogue of the Gromov compactness theorem, proven in \cite{Latremoliere15} by the second author; we quote it only for the case of {\Qqcms{(C,D)} rather than the more general {\gQqcms s} of \cite{Latremoliere15} as this suffices for our current setting.

\begin{definition}[\cite{Latremoliere15}*{Definition 4.1}]
Let $C \geq 1$ and $D\geq 0$. The \emph{covering number} $\covn{C,D}{\A,\Lip}{\varepsilon}$ of an {\Qqcms{(C,D)}} $(\A,\Lip)$, for some $\varepsilon$, is:
\begin{equation*}
\inf\left\{ \dim_\C\B : \begin{array}{l} (B,\Lip_\B)\text{ is a {\Qqcms{(C,D)}}}\\ \qpropinquity{}((\A,\Lip),(\B,\Lip_\B)) \leq \varepsilon\end{array}\right\}\text{.}
\end{equation*}
\end{definition}

\begin{theorem}[\cite{Latremoliere15}*{Theorem 4.2}]\label{tt-thm}
Let $\mathcal{A}$ be a class of {\Qqcms{(C,D)}s}, with $C\geq 1$ and $D\geq 0$, such that $\covn{C,D}{(\A,\Lip)}{\varepsilon} < \infty$ for all $\varepsilon > 0$ and $(\A,\Lip) \in \mathcal{A}$. The class $\mathcal{A}$ is totally bounded for the quantum propinquity $\qpropinquity{C,D}$ if and only if the conjunction of the the following two assertions hold:
\begin{enumerate}
\item there exists $\Delta > 0$ such that for all $(A,\Lip) \in \mathcal{A}$:
\begin{equation*}
\qdiam{\A}{\Lip} \leq \Delta\text{,}
\end{equation*}
\item there exists $G : (0,\infty) \rightarrow \N$ such that for all $(\A,\Lip)\in\mathcal{A}$ and all $\varepsilon>0$, we have:
\begin{equation*}
\covn{C,D}{\A,\Lip}{\varepsilon} \leq G(\varepsilon)\text{.}
\end{equation*}
\end{enumerate}
\end{theorem}

Our construction in Theorem (\ref{AF-lip-norms-thm}) is designed so that AF algebras with faithful tracial states are indeed limits of finite dimensional quasi-Leibniz quantum metric spaces, so we may apply Theorem (\ref{tt-thm}) to obtain:

\begin{theorem}\label{total-boundedness-thm}
If $U, L : \N \rightarrow \N\setminus\{0\}$ are two sequences in $\N\setminus\{0\}$ such that $\lim_\infty L = \lim_{\infty} U = \infty$ while $L(n)\leq  U(n)$ for all $n\in\N$, and if $k\in (0,\infty)$, then the class:
\begin{equation*}
\mathcal{AF}^k(L,U) = \left\{ (\A, \Lip_\A ) \in \mathcal{AF}^k \middle\vert \begin{array}{l}
\exists \mathcal{I} = (\A_n,\alpha_n)_{n\in\N} \quad \A = \varinjlim \mathcal{I} \\
\A_0 = \C\\
\forall n\in \N \quad L(n) \leq \dim \A_n \leq U(n)\\
\exists \mu \text{ faithful tracial state on $\A$}\quad \Lip_\A = \Lip_{\mathcal{I},\mu}^k
\end{array} \right\}
\end{equation*}
is totally bounded for the quantum propinquity.
\end{theorem}

\begin{proof}
Let $\varepsilon>0$. Let $N\in\N$ such that for all $n\geq N$ we have $L(n) \geq \sqrt[k]{\frac{1}{\varepsilon}}$. 

If $(\A,\Lip)\in\mathcal{AF}^k (L,U)$ then by definition, $\A = \varinjlim \mathcal{I}$ where $\mathcal{I} = (\A_n,\alpha_n)$ such that $U(n) \geq \dim_\C \A_n \geq L(n)$ for all $n\in\N$ and $\Lip = \Lip_{\mathcal{I},\mu}^k$ for some faithful tracial state $\mu$ of $\A$. 

Therefore, by Theorem (\ref{AF-lip-norms-thm}):
\begin{equation*}
\qpropinquity{}((\A,\Lip),(\A_N,\Lip\circ\indmor{\alpha}{N}))\leq\frac{1}{\dim(\A_N)^k}\leq \frac{1}{L(N)^k} \leq \varepsilon\text{.}
\end{equation*}
Thus $\covn{2,0}{\A,\Lip}{\varepsilon} \leq U(N)$. Moreover, $\qdiam{\A}{\Lip} \leq 2$, and thus by Theorem (\ref{tt-thm}), the class $\mathcal{AF}^k(L,U)$ is totally bounded for $\qpropinquity{}$.
\end{proof}

The quantum propinquity is not known to be complete. The dual propinquity \cite{Latremoliere13b}, introduced and studied by the second author, is a complete metric and the proper formulation of Theorem (\ref{tt-thm}) can thus be used to characterized compactness of certain classes of quasi-Leibniz compact quantum metric spaces. However, we face a few challenges when searching for compact subclasses of $\mathcal{AF}^k$. 

As the quantum propinquity dominates the dual propinquity, Theorems (\ref{AF-lip-norms-thm}), (\ref{uhf-thm}) and (\ref{af-theta-thm}) are all valid for the dual propinquity, as is Theorem (\ref{total-boundedness-thm}). However, we do not know what is the closure of the classes described in Theorem (\ref{total-boundedness-thm}) for the dual propinquity, and thus we may not conclude whether these classes are, in general, compact. It should be noted that, as shown by the second author in \cite{Latremoliere15}, there are many quasi-Leibniz quantum compact metric spaces which are limits of finite dimensional quasi-Leibniz quantum compact metric spaces for the dual propinquity.

Moreover, we do not know what the completion of the classes in Theorem (\ref{total-boundedness-thm}) are for the quantum propinquity either. Thus it is again difficult to describe compact classes from Theorem (\ref{total-boundedness-thm}).

Yet, the situation is actually quite interesting if looked at from a somewhat different perspective. Indeed, Theorems (\ref{uhf-thm}) and (\ref{af-theta-thm}) provide us with continuous maps from the Baire space to subclasses of $\mathcal{AF}^k$. Thus, knowledge about the compact subsets of $\BaireSpace$ provides actual knowledge of some compact subclasses of $\mathcal{AF}^k$ for the quantum propinquity. 

To illustrate this point, we begin by giving a theorem characterizing closed, totally bounded, and compact subspaces of the Baire space. This theorem is well-known in descriptive set theory; however the proofs of these results seem scattered in the literature and, maybe more importantly, rely on a more complex framework and terminology than is needed for our purpose. We thus include a short proof for the convenience of our readers.

\begin{notation}
If $x\in \BaireSpace$ and $n\in\N$ then we denote the finite sequence $(x_0,\ldots,x_n)$ by $x|_n$.
\end{notation}

\begin{theorem}\label{BaireSpace-top-thm}
The Baire Space $\BaireSpace$ is complete for the ultrametric $\mathsf{d}$, defined for all $x,y \in \BaireSpace$ by:
\begin{equation*}
\mathsf{d}(x,y) = 2^{-\min\{n \in \N\cup\{\infty\} : x|_n \not= y|_n \}}\text{.}
\end{equation*}
Thus the compact subsets of $\BaireSpace$ are its closed, totally bounded subsets. Moreover, for any $X\subseteq \BaireSpace$:
\begin{enumerate}
\item the closure of $X$ is the set:
\begin{equation*}
\left\{ x \in \BaireSpace : \forall n\in\N \;\exists y \in X \quad x|_n = y|_n \right\}
\end{equation*}
\item $X$ is totally bounded if and only for all $n\in\N$:
\begin{equation*}
\{ x|_n : x \in X \}
\end{equation*}
is finite.
\end{enumerate}
\end{theorem}

\begin{proof}
We prove each assertion of our theorem in each of the following step.
\setcounter{step}{0}
\begin{step}
The space $(\BaireSpace,\mathsf{d})$ is complete.
\end{step}

Let $(x^m)_{m\in\N}$ be a Cauchy sequence in $(\BaireSpace,\mathsf{d})$. For all $n\in \N$, there exists $M\in\N$ such that, if $p,q \geq M$, we have $\mathsf{d}(x^p,x^q) < \frac{1}{2^n}$. Since $\mathsf{d}$ is an ultra-metric, we have equivalently that $\mathsf{d}(x^M,x^p) < \frac{1}{2^n}$ for all $p \geq M$: thus for all $m\geq M$ we have $x^M|_n = x^p|_n$. In particular, $(x_n^m)_{m\in\N}$ is an eventually constant function for all $n\in\N$. It is then trivial to check that the sequence $(\lim_{m\rightarrow\infty} x_n^m)_{n\in\N}$ is the limit of $(x^m)_{m\in\N}$.

\begin{step}
The closure of $X\subseteq\BaireSpace$ is:
\begin{equation*}
Y = \left\{ x \in \BaireSpace : \forall n\in\N \;\exists y \in X \quad x|_n = y|_n \right\}
\end{equation*}
\end{step}

Note that by definition, $X\subseteq Y$. We now check that $Y$ is closed. Let $(z^m)_{m\in\N}$ be a sequence in $Y$ converging to some $z\in \BaireSpace$. By definition of $\mathsf{d}$, for all $N\in\N$, there exists $M\in \N$ such that for all $m\geq M$ we have $\mathsf{d}(z^m,z) < \frac{1}{2^N}$. Thus $z^M|_N = z|_N$ by definition. So $z\in Y$ as desired, and thus $Y$ is closed.

Let now $y \in Y$. Let $n\in\N$. By definition, there exists $x^n \in X$ such that $x^n|_n = y|_n$, i.e. $\mathsf{d}(x^n,y) < \frac{1}{2^n}$. Thus $(x^n)_{n\in\N}$ converges to $y$. Thus $Y$ is contained in the closure of $X$. Since $Y$ is closed, it follows from the minimality of closures that $Y$ is indeed the closure of $X$.

\begin{step}
A characterization of totally bounded subsets of the Baire Space.
\end{step}

Assume now that $X$ is totally bounded. Then for all $n \in \N$ there exists a finite subset $X_n$ of $X$ such that for all $x\in X$ there exists $y \in X_n$ with $\mathsf{d}(x,y)< \frac{1}{2^n}$, or equivalently, such that $x|_n = y|_n$. Thus $\{x|_n : x\in X\} = \{x|_n : x\in X_n\}$, the latter being finite. Conversely, note that $X_n$ converges to $X$ for the Hausdorff distance $\Haus{\mathsf{d}}$, and thus if $(X_n)_{n\in\N}$ is finite for all $n\in\N$, we conclude easily that $X$ is totally bounded.
\end{proof}

\begin{remark}
Theorem (\ref{BaireSpace-top-thm}) is well-known in descriptive set theory, though the proof is often presented within a much more elaborate framework. Our assertion about the closure of sets is often phrased by noting that a subset of $\BaireSpace$ is closed if and only if it is given as all infinite paths in a pruned tree. In this context, a tree over the Baire Space is a subset of the collection of all \emph{finite} sequences valued in $\N\setminus\{0\}$ with a simple hereditary property: if a finite sequence is in our tree, then so is its sub-sequence obtained by dropping the last entry. A pruned tree is a tree $T$ such that every sequence in it is a proper sub-sequence of another element of $T$. Last, a path is simply a sequence $x\in\BaireSpace$ such that $x|_n \in T$ for all $N$. This relation makes the translation between Theorem (\ref{BaireSpace-top-thm}) and the terminology of certain branches of set theory.

Moreover, a tree is finitely branching when given a finite sequence $x$ of length $n$ in the tree, there are only finitely many possible finite sequences of length $n+1$ whose $n$ first entries coincide with $x$. It is easy to see that Theorem (\ref{BaireSpace-top-thm}) exactly states that a subset of the Baire space is compact if and only if it consists of all infinite paths through a pruned tree with finite branching (and our theorem makes the tree explicit)
\end{remark}

We now apply Theorem (\ref{BaireSpace-top-thm}) to identify certain compact subclasses of UHF algebras and Effros-Shen AF algebras.

\begin{corollary}\label{uhf-compact-corollary}  
For all $k\in\N$ and all sequence $B : \N\rightarrow \N\setminus\{0\}$ with $\sqrt{\frac{B(n+1)}{B(n)}} \in \N\setminus\{0,1\}$ for all $n\in\N$, the class:
\begin{equation*}
\mathcal{UHF}^k \cap \mathcal{AF}^k((2^n)_{n\in\N}, B)
\end{equation*}
is compact for the quantum propinquity $\qpropinquity{}$.
\end{corollary}

\begin{proof}
Let:
\begin{equation*}
X = \left\{ x \in \BaireSpace : \forall n\in \N\quad x_n + 1 \leq \sqrt{\frac{B(n+1)}{B(n)}}\right\}\text{.}
\end{equation*}
By construction, $\uhf{X,k} = \mathcal{UHF}^k \cap \mathcal{AF}^k((2^n)_{n\in\N}, B)$ (the lower bound on the dimension of the matrix algebras was observed in the proof of Theorem (\ref{uhf-thm})).  On the other hand, by Theorem (\ref{BaireSpace-top-thm}), the set $X$ is compact and by Theorem (\ref{uhf-thm}), the map $\uhf{\cdot,k}$ is continuous. So $\mathcal{UHF}^k \cap \mathcal{AF}^k((2^n)_{n\in\N}, B)$ is compact.
\end{proof}

We also obtain:

\begin{corollary}\label{af-theta-compact-corollary}
Let $C, B \in \BaireSpace$, and set:
\begin{equation*}
X = \left\{ \theta \in (0,1)\setminus\Q : \theta = \lim_{n\rightarrow\infty} \cfrac{1}{r_1 + \cfrac{1}{r_2 + \cfrac{1}{\cdots + \cfrac{1}{r_n}}}} \text{ and }\forall n\in \N \quad C(n) \leq r_n \leq B(n) \right\}
\end{equation*}
Then the set:
\begin{equation*}
\left\{ (\A,\Lip) \in \mathcal{AF}^k : \A\in \af{X} \right\}
\end{equation*}
is compact for the quantum propinquity $\qpropinquity{}$.
\end{corollary}

\begin{proof}
This follows from Theorem (\ref{BaireSpace-top-thm}) and the continuity established in Theorem (\ref{af-theta-thm}).
\end{proof}

We were thus able to obtain several examples of compact classes of quasi-Leibniz quantum compact metric spaces for the quantum propinquity and consisting of infinitely many AF algebras, which is a rather notable result. We also note that since the dual propinquity \cite{Latremoliere13b} is also a metric up to isometric isomorphism and is dominated by the quantum propinquity, the topology induced by the quantum propinquity and the dual propinquity on these compact classes must agree.

\section{Quantum ultrametrics on the Cantor Set}\label{cantor}

The Gel'fand spectrum of Abelian AF algebras are homeomorphic to compact subspaces of the Cantor set. In this section, we will explore the Monge-Kantorovich metrics induced by the Lip-norms defined in Theorem (\ref{AF-lip-norms-thm}) on the Cantor set itself. We will prove, in particular, that the standard ultrametrics on the Cantor set can be recovered directly from our construction.

There are many standard presentations of the Cantor set, and we shall pick the following for our purpose:
\begin{notation}
Let $\Z_2 = \{0,1\}$ with the discrete topology. The Cantor set is given by:
\begin{equation*}
\mathcal{C} = \left\{ (z_n)_n\in \N : n \in \Z_2 \right\} = \prod_{n\in\N} \Z_2
\end{equation*}
with the product topology.
\end{notation}

In order to fit the Cantor set inside the framework of this paper, we shall explicit a natural inductive sequence of finite dimensional Abelian C*-algebras converging to the C*-algebra $C(\mathcal{C})$ of $\C$-valued continuous functions on $\mathcal{C}$.

\begin{notation}\label{Cantor-af-notation}
For all $n\in\N$, we denote the evaluation map $(z_m)_{m\in\N}\in\mathcal{C} \mapsto z_n$ by $\eta_n$. Note that $\eta_n \in C(\mathcal{C})$ is a projection and $u_n = 2\eta_n - \unit_{C(\mathcal{C})}$ is a self-adjoint unitary in $C(\mathcal{C})$. 

We set $\A_0 = \C \unit_{C(\mathcal{C})}$ and, for all $n\in\N\setminus\{0\}$, we set:
\begin{equation*}
\A_n = \mathrm{C^\ast}\left( \{ \unit_{C(\mathcal{C})} , u_0, \ldots, u_{n-1}\}\right)\text{.}
\end{equation*}

By definition, $\A_n$ is a finite dimensional C*-subalgebra of $C(\mathcal{C})$, with the same unit as $C(\mathcal{C})$. Moreover, $\A_n \subseteq \A_{n+1}$ for all $n\in\N$. Last, it is easy to check that $\bigcup_{n\in\N}\A_n$ is a unital *-subalgebra of $C(\mathcal{C})$ which separates points; as $\mathcal{C}$ is compact, the Stone-Weierstra{\ss}   theorem implies that:
\begin{equation*}
C(\mathcal{C}) = \mathrm{closure}\left(\bigcup_{n\in\N} \A_n\right)\text{.}
\end{equation*}
If we denote the inclusion map $\A_n \hookrightarrow \A_{n+1}$ by $\alpha_n$ for all $n\in\N$, then $C(\mathcal{C}) = \varinjlim\mathcal{T}$ where $\mathcal{T} = (\A_n,\alpha_n)$. We note that of course, $\indmor{\alpha}{n}$ is just the inclusion map of $\A_n$ into $C(\mathcal{C})$ for all $n\in\N$; whenever possible we will thus omit the maps $\alpha_n$ and $\indmor{\alpha}{n}$ from our notations.
\end{notation}

We now have our standard description of $C(\mathcal{C})$ as an AF algebra, and a specific inductive sequence to use in Theorem (\ref{AF-lip-norms-thm}). We also require a particular choice of a faithful tracial state; as $C(\mathcal{C})$ is Abelian, we have quite some choice of such states. We will focus our attention on a specific construction.

\begin{notation}\label{Cantor-trace-notation}
The set $\mathcal{C} = \prod_{n\in\N} \Z_2$ is a group for the pointwise addition modulo 1. As $\mathcal{C}$ is compact, there exists a unique Haar probability measure, which defines by integration a faithful tracial state $\lambda$ on $C(\mathcal{C})$. 

It is easy to check that, for any finite, nonempty $F\subset \N$, we have:
\begin{equation*}
\lambda\left(\prod_{j\in F}\eta_j\right) = 2^{-\# F}
\end{equation*}
where $\# F$ is the cardinal of $F$. Indeed, $\prod_{j\in F}\eta_j$ is simply the indicator function of the subset:
\begin{equation*}
\left\{ (z_n)_{n\in\N} \in \mathcal{C} : \forall j \in F\quad z_j = 1\right\}\text{.}
\end{equation*}
It is then easy to check that $\mathcal{C}$ is the union of $2^{\# F}$ disjoint translates  of $F$.
\end{notation}

The primary advantage of our choice of tracial state is illustrated in the following lemma.

\begin{lemma}\label{Cantor-ortho-lemma}
We shall use Notations (\ref{Cantor-af-notation}) and (\ref{Cantor-trace-notation}). If we endow $C(\mathcal{C})$ with the inner product:
\begin{equation*}
(f,g) \in C(\mathcal{C}) \mapsto \lambda(f\overline{g})\text{,}
\end{equation*}
then $u_{n} \in \A_n^\perp$ for all $n\in\N$. Moreover $\left(\prod_{j \in F} u_j\right)_{F\in \mathcal{F}}$, where $\mathcal{F}$ is the set of nonempty finite subsets of $\N$, is an orthonormal family of $L^2(C(\mathcal{C}),\lambda)$.
\end{lemma}

\begin{proof}
We let, for all $n\in\N\setminus\{0\}$:
\begin{equation*}
\B_n = \left\{ \unit_{C(\mathcal{C})}, \prod_{j \in F} u_j : F \text{ is a nonempty subset of $\{0,\ldots,n-1\}$} \right\}\text{.}
\end{equation*}
We note that $\B_n$ is a basis for $\A_n$. We also note that $\left(\prod_{j \in F} u_j\right)_{F\in \mathcal{F}}$ is a Hamel basis of the space $\bigcup_{n\in\N} \A_n$.

Now, let $n\in\N$ and $F \subseteq \{0,\ldots,n-1\}$ be nonempty. We have:
\begin{equation}\label{Cantor-orth-lemma-eq1}
\begin{split}
\lambda\left( u_n^\ast \prod_{j \in F} u_j\right) &= \lambda\left( (2\eta_n-\unit_{C(\mathcal{C})}) \prod_{j\in F}(2\eta_j - \unit_{C(\mathcal{C})}) \right)\\
&= \lambda\left( \prod_{j\in F\cup\{n\}}(2\eta_j - \unit_{C(\mathcal{C})}) \right)\\
&= \sum_{G \subseteq F\cup\{n\}} (-1)^{\# F + 1 - \# G} 2^{\# G} \lambda\left(\prod_{j \in G}\eta_j\right)\\
&= \sum_{G \subseteq F\cup\{n\}} (-1)^{\# F + 1 - \# G} 2^{\# G} 2^{-\# G}\\
&= \sum_{j \in F\cup\{n\} } {\#F + 1 \choose j} (-1)^j \\
&= (1 - 1)^{\# F + 1} = 0 \text{.}
\end{split}
\end{equation}
Since $\alg{B}_n$ is a basis for $\A_n$, we conclude that indeed, $u_n \in \A_n^\perp$.

Moreover, we note that Expression (\ref{Cantor-orth-lemma-eq1}) also proves that $\alg{B}$ is an orthogonal family in $L^2(C(\mathcal{C}),\lambda)$. As the product of unitaries is unitary, our definition of the inner product then shows trivially that the family $\alg{B}$ is orthonormal.
\end{proof}

We now have the tools needed to state our main theorem for this section: Lip-norms defined using Theorem (\ref{AF-lip-norms-thm}) with the ingredients described in this section naturally lead to ultrametrics on the Cantor space via the associated Monge-Kantorovich metric.

\begin{theorem}\label{Cantor-ultra-thm}
Let $\beta : \N \rightarrow \N\setminus\{0\}$ be a decreasing sequence with $\lim_\infty \beta = 0$. Identifying the Cantor space $\mathcal{C}$ with the Gel'fand spectrum of $C(\mathcal{C})$, and using Notations (\ref{Cantor-af-notation}) and (\ref{Cantor-trace-notation}), we have, for all $x,y \in \mathcal{C}$:
\begin{equation*}
\Kantorovich{\Lip_{\mathcal{T},\lambda}^\beta} (x,y) = 
\begin{cases}
0 \text{ if $x = y$,}\\
2\beta(\min\{n \in \N : x_n \not= y_n\}) \text{ otherwise.}
\end{cases}
\end{equation*}
By construction, $\Kantorovich{\Lip_{\mathcal{T},\lambda}^\beta}$ is an ultrametric on $\mathcal{C}$.
\end{theorem}

\begin{proof} 
In this proof, we will denote by $\CondExp{\cdot}{\A_n}$ the conditional expectation from $C(\mathcal{C})$ onto $\A_n$, which leave $\lambda$ invariant.

Fix $x \neq y \in \mathcal{C}$. By Theorem (\ref{AF-lip-norms-thm}):
\begin{equation*}
\Kantorovich{\Lip_{\mathcal{T},\lambda}^\beta}(x,y) =\sup \left\{ \vert f(x) - f(y) \vert : f \in C(\mathcal{C}), \Lip_{\mathcal{T},\lambda}^\beta(f)\leq 1  \right\}\text{.}
\end{equation*}

Our computation relies on the following observation. Let $n \geq k \in \N$. Since $u_n \in \A_k^\perp$ in $L^2(C(\mathcal{C}),\lambda)$ by Lemma (\ref{Cantor-ortho-lemma}), we conclude that $\CondExp{u_n}{\A_k} = 0$. Of course, if $k > n \in \N$ then $\CondExp{u_n}{\A_k} = u_n$. Thus we have for all $n\in\N$:
\begin{equation*}
\begin{split}
\Lip_{\mathcal{T},\lambda}^\beta (u_n) &= \max\left\{ \frac{\left\| u_n \right\|_{C(\mathcal{C})}}{\beta(k)} : k \leq n  \right\}\\
&= \max\left\{ \frac{1}{\beta(k+1)} : k < n  \right\}\text{ as $u_n$ is unitary,}\\
&= \frac{1}{\beta(n)}\text{ as $\beta$ is decreasing.}
\end{split}
\end{equation*}

We thus have $\Lip_{\mathcal{T},\lambda}^\beta (\beta(n)u_n) \leq 1$ for all $n\in\N$.

Let $N = \min\{ n\in\N : x_n \not= y_n \}$. Then $|u_N(x) - u_N(y)| = 2$ and thus:
\begin{equation*}
\begin{split}
\Kantorovich{\Lip_{\mathcal{T},\lambda}^\beta}(x,y) \geq \beta(N) |u_N(x) - u_N(y)| = 2\beta(N)\text{.} 
\end{split}
\end{equation*}

On the other hand, for all $n\in\N$, the C*-algebra $\A_n$ is the C*-subalgebra generated by the evaluation maps $\eta_j$ for $j = 0,\ldots,n$ and the identity. Therefore, for any $f\in \mathcal{C}$, we have $\CondExp{f}{\A_n}(x) = \CondExp{f}{\A_n}(y)$ for all $n < N$.

Let $f\in C(\mathcal{C})$ with $\Lip_{\mathcal{T},\lambda}^\beta(f) \leq 1$. Then for all $n < N$:
\begin{equation*}
\begin{split}
|f(x) - f(y)| &= |f(x) - \CondExp{f}{\A_n}(x) - (f(y) - \CondExp{f}{\A_n}(y)) |\\
&\leq 2\|f - \CondExp{f}{\A_n} \|_{C(\mathcal{C})} \\
&\leq 2\beta(n+1)\text{.}
\end{split}
\end{equation*}
Since $\beta$ is decreasing, we thus get:
\begin{equation*}
|f(x) - f(y)| \leq 2 \min\{\beta(n+1) : n < N\} = 2 \beta(N)\text{.}
\end{equation*}
We thus conclude that:
\begin{equation*}
\Kantorovich{\Lip_{\mathcal{T},\lambda}^\beta}(x,y) = 2\beta(N) \text{,}
\end{equation*}
as desired. It is easy to check that $\Kantorovich{\Lip_{\mathcal{T},\lambda}^\beta}$ defines an ultrametric on $\mathcal{C}$ since $\beta$ is decreasing.
\end{proof}

We thus recognize standard ultrametrics on the Cantor set:

\begin{corollary}\label{Cantor-metric-corollary}
Let $r > 1$, and set $\beta_r : n \in \N \mapsto \frac{1}{2} r^{-n}$. Then, for any two $x,y \in \mathcal{C}$, using the notations of Theorem (\ref{Cantor-ultra-thm}), we have:
\begin{equation*}
\Kantorovich{\Lip_{\mathcal{T},\lambda}^{\beta_r}}(x,y) = 
\begin{cases}
0 \text{ if $x=y$,}\\
r^{-\min\{ n \in \N : x_n\not= y_n \} }\text{ otherwise.}
\end{cases}
\end{equation*}
\end{corollary}

\section{Family of Lip-norms for a fixed AF algebra}\label{family-lip-norms}

In this section, we consider the situation in which we fix a unital AF-algebra with faithful tracial state and consider the construction of the Lip-norm from Theorem (\ref{AF-lip-norms-thm}), in which we vary our choices of the sequence $\beta$.  From this, we describe convergence in quantum propinquity with respect to this notion.  We note that Section (\ref{af-theta}) essentially provides an outline for the process. 

\begin{notation}
Let $\beta : \N \rightarrow (0, \infty) $ be a positive sequence that tends to $0$ at infinity.  Denote the space of real-valued sequences that converge to $0$ as $c_0 (\N, \R)$. Define:
\begin{equation*}
c_\beta = \left\{ x \in c_0 (\N, \R) : \forall n \in \N , \ 0<x(n) \leq \beta(n) \right\}.  
\end{equation*}
\end{notation}

\begin{theorem}\label{family-lip-norms-thm} Let $\A$ be an AF algebra endowed with a faithful tracial state $\mu$. Let $\mathcal{I} = (\A_n,\alpha_n)_{n\in\N}$ be an inductive sequence of finite dimensional C*-algebras with C*-inductive limit $\A$, with $\A_0 = \C$ and where $\alpha_n$ is unital for all $n\in\N$.  If $\beta: \N \rightarrow (0, \infty)$ is a positive sequence that  tends to $0$ at infinity and  $\left(x^k \right)_{k \in \N} \cup \{ x\} \subset c_\beta $ such that $x^k $ converges point-wise to $x$, then using the notations of Theorem (\ref{AF-lip-norms-thm}):
\begin{equation*}
\lim_{ k \to \infty} \qpropinquity{}\left(\left( \A, \Lip^{x^k }_{\mathcal{I}, \mu}\right), \left( \A, \Lip^{x}_{\mathcal{I}, \mu}\right)\right) = 0.
\end{equation*}
\end{theorem}
\begin{proof}  The proof follows the procedure from Section (\ref{af-theta}). We begin by verifying some details.

Let $\beta: \N \rightarrow (0, \infty)$ be a positive sequence that  tends to $0$ at infinity.

Assume that $\left(x^k \right)_{k \in \N} \cup \{ x\} \subset c_\beta $ such that $x^k $ converges point-wise to $x$.
Next, we show convergence of the finite dimensional spaces $\A_n $ for all $n \in \N$.  Thus, fix $N \in \N$.  Let $y \in c_\beta$, so that $y(n)>0$ for all $n \in \N$,  and let $a \in \A_N $.  Then: 
\begin{equation*}
\Lip^y_{\mathcal{I}, \mu} \circ \indmor{\alpha}{N} (a) = \max \left\{ \frac{\left\Vert \indmor{\alpha}{N}(a) - \CondExp{\indmor{\alpha}{N}(a)}{\indmor{\alpha}{n}(\A_n)}\right\Vert_\A}{y(n)} : n \in \N , n \leq N \right\}.
\end{equation*}

Define $R^N_+ = \{ y=(y(0), y(1), \ldots, y(N)) \in \R^{N+1} : \forall n \in \{0,1,\ldots, N\} , y(n)>0 \}$. For $x,y \in R^N_+$, we define $\mathsf{d}_\infty (x,y) = \max \left\{ \vert x(n) -y(n)\vert : n \in \{0,1,\ldots , N \}\right\} $.  Thus, $\left(R^N_+ , \mathsf{d}_\infty \right)$ is a metric space. 
Define $g :   R^N_+ \times \A_N \rightarrow \R$ by: 
\begin{equation*}
g(y,a) =  \max \left\{ \frac{\left\Vert \indmor{\alpha}{N}(a) - \CondExp{\indmor{\alpha}{N}(a)}{\indmor{\alpha}{n}(\A_n)}\right\Vert_\A}{y(n)} : n \in \N , n \leq N \right\},
\end{equation*}
which is finite by definition of $R^N_+$. 
Therefore, it follows that:
\begin{equation*}
g : \left(R^N_+ , \mathsf{d}_\infty \right) \times \left(\A_N , \Vert \cdot \Vert_{\A_N} \right) \rightarrow \R
\end{equation*}
is continuous.  Denote the class of all $(2,0)$-quasi-Leibniz quantum compact metric spaces by $\mathcal{QQCMS}_{2,0}$.  Next, define 
$G :  R^N_+ \rightarrow \mathcal{QQCMS}_{2,0} $ by: 
\begin{equation*}
G(y)= (\A_N , g(y, \cdot ) ),
\end{equation*}
which is well-defined by definition of $g$.  
Thus,  following the proof of Theorem (\ref{af-theta-fd-lemma}), we conclude that $G :   \left(R^N_+ , \mathsf{d}_\infty \right) \rightarrow \left(\mathcal{QQCMS}_{2,0} , \qpropinquity{} \right)$ is continuous.  If $ y \in \R^\N$, then we denote $y|_N = (y(0) , y(1) , \ldots , y(N) ) $.  Since $\left(x^k \right)_{k \in \N} \cup \{ x\} \subset c_\beta $, we have that $\left(x^k|_N \right)_{k \in \N} \cup \left\{ x|_N\right\} \subset R^N_+. $ Furthermore, the assumption that $x^k $ converges pointwise to $x$  implies that $\lim_{k \to \infty} \mathsf{d}_\infty \left( x^k|_N , x|_N \right) =0 $.  Therefore:
\begin{equation*}
\lim_{k \to \infty} \qpropinquity{}\left(G\left(x^k|_N \right), G\left(x|_N \right)\right)=0 .
\end{equation*}  
But, for all $k \in \N$:
\begin{equation*}
\qpropinquity{}\left(\left(\A_N , \Lip^{x^k}_{\mathcal{I}, \mu} \circ \indmor{\alpha}{N}\right), \left(\A_N ,  \Lip^{x}_{\mathcal{I}, \mu} \circ \indmor{\alpha}{N}\right)\right)=\qpropinquity{}\left(G\left(x^k|_N \right), G\left(x|_N \right)\right). 
\end{equation*}
We thus have:
\begin{equation}\label{finite-beta}
\lim_{k \to \infty} \qpropinquity{}\left(\left(\A_N , \Lip^{x^k}_{\mathcal{I}, \mu} \circ \indmor{\alpha}{N}\right), \left(\A_N ,  \Lip^{x}_{\mathcal{I}, \mu} \circ \indmor{\alpha}{N}\right)\right)=0.
\end{equation}
As $N \in \N$ was arbitrary, we conclude that Equation (\ref{finite-beta}) is true for all $n \in \N $.

We are now ready to prove convergence.   Let $\varepsilon >0 $.   There exists $M \in \N $ such that for all $n \geq M $ , $\beta(n) < \varepsilon/2 $.  Hence, if $n \geq M $, then by Theorem (\ref{AF-lip-norms-thm}) and definition of $c_\beta$:
\begin{equation*}
\mathsf{\Lambda} \left(\left(\A_n ,  \Lip^{x^k}_{\mathcal{I}, \mu} \circ \indmor{\alpha}{n}\right), \left(\A ,  \Lip^{x^k}_{\mathcal{I}, \mu}\right)\right) \leq x^k (n) \leq \beta(n) < \varepsilon/2
\end{equation*}
for all $k \in \N$ and:
\begin{equation*}
\mathsf{\Lambda} \left(\left(\A_n ,  \Lip^{x}_{\mathcal{I}, \mu} \circ \indmor{\alpha}{n}\right), \left(\A , \Lip^{x}_{\mathcal{I}, \mu}\right)\right) \leq x (n) \leq \beta(n) < \varepsilon/2 .
\end{equation*}

By the triangle inequality and Equation (\ref{finite-beta}), we thus get:
\begin{equation*}
\lim \sup_{k \to \infty} \mathsf{\Lambda} \left( \left(\A , \Lip^{x^k}_{\mathcal{I}, \mu}\right), \left(\A , \Lip^{x}_{\mathcal{I}, \mu} \right)\right) \leq \varepsilon.
\end{equation*}
As $\varepsilon>0$ was arbitrary, $\lim_{k \to \infty} \mathsf{\Lambda} \left( \left(\A , \Lip^{x^k}_{\mathcal{I}, \mu} \right), \left(\A , \Lip^{x}_{\mathcal{I}, \mu} \right)\right)=0$.
\end{proof}

In particular, for the Cantor set, we can use this result to discuss continuity in quantum propinquity of the continuous functions on the Cantor set with respect to the quantum ultrametrics discussed in Section (\ref{cantor}).  All that is required is a sequence in $c_\beta $, which converges point-wise to some element in $ c_\beta $.  We present this in the case of the standard ultrametrics, and note that although we are using the same $C^\ast$-algebra, $C(\mathcal{C})$, if $ r \neq s$, then the associated standard ultrametrics on the Cantor set are not isometric.  This implies that the function defined in the following Corollary (\ref{cantor-beta}) is not constant up to isometric isomorphism.

\begin{corollary}\label{cantor-beta}Let $r>1$, and set $\beta_r : n \in \N \mapsto  \frac{1}{2}r^{-n}$.   Using the notations of Theorem (\ref{AF-lip-norms-thm}) along with Notations (\ref{Cantor-af-notation}) and (\ref{Cantor-trace-notation}), the function: 
\begin{equation*}
\mathsf{u}: r \in (1, \infty) \longmapsto \left(C(\mathcal{C}), \Lip^{\beta_r}_{\mathcal{T},\lambda}\right)
\end{equation*}
is continuous from $(1, \infty)$ to the class of $(2,0)$-quasi-Leibniz quantum compact metric spaces metrized by the quantum propinquity $\mathsf{\Lambda}$.
\end{corollary}

\begin{proof}
Let  $(r_n)_{n \in \N} \cup \{r\} \subset (1,\infty)$ such that $\lim_{n \to \infty} \vert r_n -r|=0$. Since  $(r_n)_{n \in \N} \cup \{r\}$ is a compact set, there exists some $a>1$ such that for all $n \in \N $, $r_n , r \in [a,  \infty)$.  Therefore, for all $n \in \N $, we have that $\beta_{r_n} , \beta_r \in c_{\beta_a } $.  The sequence $(\beta_{r_n})_{n \in \N}$ converges point-wise to $\beta_r $ since:
\begin{equation*}
 \lim_{n \to \infty } \vert \beta_{r_n} (m) - \beta_r (m)\vert = \lim_{n \to \infty} \vert \tfrac{1}{2}r_n^{-m} - \tfrac{1}{2}r^{-m} \vert =0 
\end{equation*} for all $ m \in \N$.  Hence, by the Theorem (\ref{family-lip-norms-thm}),
\begin{equation*}
\lim_{n \to \infty} \mathsf{\Lambda} (\mathsf{u}(r_n) , \mathsf{u}(r))=0.
\end{equation*}  
Thus, sequential continuity provides the desired result.  
\end{proof}

\providecommand{\bysame}{\leavevmode\hbox to3em{\hrulefill}\thinspace}
\providecommand{\MR}{\relax\ifhmode\unskip\space\fi MR }
\providecommand{\MRhref}[2]{%
  \href{http://www.ams.org/mathscinet-getitem?mr=#1}{#2}
}
\providecommand{\href}[2]{#2}

\vfill


\begin{thebibliography}{10}

\normalsize
\baselineskip=17pt


\bibitem{Antonescu04}
{C}. {A}ntonescu and {E}. {C}hristensen, \emph{Spectral triples for af
  {$C^\ast$}-algebras and metrics on the cantor set}, J. Oper. Theory
  \textbf{56} (2006), no.~1, 17--46, ArXiv: 0309044.

\bibitem{Blackadar80}
{B}. {B}lackadar, \emph{Traces on simple {AF} {$C^\ast$-algebras}}, Journal of
  Funct. Anal. \textbf{38} (1980), 156--168.

\bibitem{Bratteli72}
{O}. {B}ratteli, \emph{Inductive limits of finite dimensional
  {$C^\ast$-algebras}}, Trans. Amer. Math. Soc. \textbf{171} (1972), 195--234.

\bibitem{Bratteli79}
{O}. {B}ratteli and {D}. {R}obinson, \emph{Operator algebras and quantum
  statistical mechanics i}, Springer-Verlag, 1979.

\bibitem{Connes89}
A.~{C}onnes, \emph{Compact metric spaces, {F}redholm modules and
  hyperfiniteness}, Ergodic Theory and Dynamical Systems \textbf{9} (1989),
  no.~2, 207--220.

\bibitem{Connes}
\bysame, \emph{Noncommutative geometry}, Academic Press, San Diego, 1994.

\bibitem{Davidson}
{K}.~{R}. {D}avidson, \emph{{C*}--algebras by example}, Fields Institute
  Monographs, American Mathematical Society, 1996.

\bibitem{Effros80b}
{E}.~{G}. {E}ffros and {C}.~{L}. {S}hen, \emph{Approximately finite
  {$C^\ast$}-algebras and continued fractions}, Indiana University Mathematics
  Journal \textbf{29} (1980), no.~2, 191--204.

\bibitem{Glimm60}
{J}. {G}limm, \emph{On a certain class of operator algebras}, Trans. Amer.
  Math. Soc. \textbf{95} (1960), 318--340.

\bibitem{Goodearl77}
{K}. {G}oodearl, \emph{Algebraic representations of {C}hoquet simplexes}, J.
  Pure Appl. Algebra \textbf{11} (1977), 111--130.

\bibitem{Gromov81}
M.~{G}romov, \emph{Groups of polynomial growth and expanding maps},
  Publications math{\'e}matiques de l' {I. H. E. S.} \textbf{53} (1981),
  53--78.

\bibitem{Gromov}
\bysame, \emph{Metric structures for {R}iemannian and non-{R}iemannian spaces},
  Progress in Mathematics, Birkh{\"a}user, 1999.

\bibitem{Hardy38}
{G}.~{H}. {H}ardy and {E}.~{M}. {W}right, \emph{An introduction to the theory
  of numbers}, fourth ed., Oxford University Press, 1938, 1975.

\bibitem{Hausdorff}
{F}. {H}ausdorff, \emph{{G}rundz{\"u}ge der {M}engenlehre}, Verlag Von Veit und
  Comp., 1914.

\bibitem{kerr02}
D.~{K}err, \emph{Matricial quantum {G}romov-{H}ausdorff distance}, J. Funct.
  Anal. \textbf{205} (2003), no.~1, 132--167, math.OA/0207282.

\bibitem{Latremoliere05}
{F}. {L}atr{\'e}moli{\`e}re, \emph{Approximation of the quantum tori by finite
  quantum tori for the quantum gromov-hausdorff distance}, Journal of Funct.
  Anal. \textbf{223} (2005), 365--395, math.OA/0310214.

\bibitem{Latremoliere12b}
\bysame, \emph{Quantum locally compact metric spaces}, Journal of Functional
  Analysis \textbf{264} (2013), no.~1, 362--402, ArXiv: 1208.2398.

\bibitem{Latremoliere13c}
\bysame, \emph{Convergence of fuzzy tori and quantum tori for the quantum
  {G}romov--{H}ausdorff {P}ropinquity: an explicit approach.}, Accepted,
  M{\"u}nster Journal of Mathematics (2014), 41 pages, ArXiv: math/1312.0069.

\bibitem{Latremoliere14}
\bysame, \emph{The triangle inequality and the dual {G}romov-{H}ausdorff
  propinquity}, Accepted in Indiana University Journal of Mathematics (2014),
  16 Pages., ArXiv: 1404.6633.

\bibitem{Latremoliere15c}
\bysame, \emph{Curved noncommutative tori as leibniz compact quantum metric
  spaces}, Journal of Math. Phys. \textbf{56} (2015), no.~12, 123503, 16 pages,
  ArXiv: 1507.08771.

\bibitem{Latremoliere13b}
\bysame, \emph{The dual {G}romov--{H}ausdorff {P}ropinquity}, Journal de
  Math{\'e}matiques Pures et Appliqu{\'e}es \textbf{103} (2015), no.~2,
  303--351, ArXiv: 1311.0104.

\bibitem{Latremoliere15b}
\bysame, \emph{Quantum metric spaces and the {G}romov-{H}ausdorff propinquity},
  Accepted in Contemp. Math. (2015), 88 pages, ArXiv: 150604341.

\bibitem{Latremoliere15}
\bysame, \emph{A compactness theorem for the dual {G}romov-{H}ausdorff
  propinquity}, Accepted in Indiana University Journal of Mathematics (2016),
  40 Pages, ArXiv: 1501.06121.

\bibitem{Latremoliere13}
\bysame, \emph{The {Q}uantum {G}romov-{H}ausdorff {P}ropinquity}, Trans. Amer.
  Math. Soc. \textbf{368} (2016), no.~1, 365--411, electronically published on
  May 22, 2015, http://dx.doi.org/10.1090/tran/6334, ArXiv: 1302.4058.

\bibitem{Lin2001}
{H}. {L}in, \emph{An introduction to the classification of amenable
  {$C^\ast$}-algebras}, World Scientific, 2001.

\bibitem{Lin08}
\bysame, \emph{{AF}-embedding of crossed product of {AH} algebras by {$\Z$} and
  asymptotic {AF} embedding}, Indiana University Mathematics Journal
  \textbf{57} (2008), no.~2, 891--944, ArXiv: math/0612529.

\bibitem{Miller95}
{A}.~{W}. {Miller}, \emph{Descriptive set theory and forcing: how to prove
  theorems about borel sets the hard way}, Springer-Verlag, 1995.

\bibitem{Ozawa05}
{N}. {O}zawa and M.~A. {R}ieffel, \emph{Hyperbolic group {$C\sp\ast$}-algebras
  and free product {$C\sp\ast$}-algebras as compact quantum metric spaces},
  Canad. J. Math. \textbf{57} (2005), 1056--1079, ArXiv: math/0302310.

\bibitem{PimVoi80a}
{M}. {P}imsner and {D}.~{V}. Voiculescu, \emph{Imbedding the irrational
  rotation algebras into an {AF} algebra}, Journal of Operator Theory
  \textbf{4} (1980), 201--210.

\bibitem{Rieffel98a}
M.~A. {R}ieffel, \emph{Metrics on states from actions of compact groups},
  Documenta Mathematica \textbf{3} (1998), 215--229, math.OA/9807084.

\bibitem{Rieffel99}
\bysame, \emph{Metrics on state spaces}, Documenta Math. \textbf{4} (1999),
  559--600, math.OA/9906151.

\bibitem{Rieffel01}
\bysame, \emph{Matrix algebras converge to the sphere for quantum
  {G}romov--{H}ausdorff distance}, Mem. Amer. Math. Soc. \textbf{168} (2004),
  no.~796, 67--91, math.OA/0108005.

\bibitem{Rieffel09}
\bysame, \emph{Distances between matrix alegbras that converge to coadjoint
  orbits}, Proc. Sympos. Pure Math. \textbf{81} (2010), 173--180, ArXiv:
  0910.1968.

\bibitem{Rieffel10c}
\bysame, \emph{{L}eibniz seminorms for "matrix algebras converge to the
  sphere"}, Clay Math. Proc. \textbf{11} (2010), 543--578, ArXiv: 0707.3229.

\bibitem{Rieffel11}
\bysame, \emph{Leibniz seminorms and best approximation from
  {$C^\ast$}-subalgebras}, Sci China Math \textbf{54} (2011), no.~11,
  2259--2274, ArXiv: 1008.3773.

\bibitem{Rieffel12}
\bysame, \emph{Standard deviation is a strongly {L}eibniz seminorm}, Submitted
  (2012), 24 pages, ArXiv: 1208.4072.

\bibitem{Rieffel15}
\bysame, \emph{Matricial bridges for "matrix algebras converge to the sphere"},
  Submitted (2015), 31 pages, ArXiv: 1502.00329.

\bibitem{Rieffel00}
\bysame, \emph{{G}romov-{H}ausdorff distance for quantum metric spaces}, Mem.
  Amer. Math. Soc. \textbf{168} (March 2004), no.~796, math.OA/0011063.

\bibitem{weaver99}
{N}. {W}eaver, \emph{Lipschitz algebras}, World Scientific, 1999.

\end{thebibliography}
\end{document}